\documentclass[a4paper,11pt,reqno,english]{smfart}
 
\usepackage{psfrag}
\usepackage{latexsym}
\usepackage{array}
\usepackage{amssymb}
\theoremstyle{plain}

\usepackage{color}

\usepackage{amsmath,amssymb,amsfonts,graphicx,psfrag}
\usepackage[latin1]{inputenc}

\usepackage[colorlinks=true, pdfstartview=FitV, linkcolor=blue, citecolor=blue, urlcolor=blue]{hyperref}
\newtheorem{thm}{Theorem}[section]
\newtheorem{prop}[thm]{Proposition}
\newtheorem{lem}[thm]{Lemma}
\newtheorem{df}[thm]{Definition}

\newtheorem{rem}[thm]{Remark}
\newtheorem{ass}[thm]{Assumption}
\newtheorem{cor}[thm]{Corollary}

\def\be#1 {\begin{equation} \label{#1}}
\newcommand{\ee}{\end{equation}}

\newcommand{\mb}{\medskip\noindent}
\newcommand{\gb}{\bigskip\noindent}
\newcommand{\R}{\mathbb R}

\newcommand{\I}{\mathcal I}

\newcommand{\RR}{\mathrm R}

\newcommand{\PPP}{\mathrm P}

\def \OO {\mathcal{O}}

\def \BB {{\mathcal{B}}}

\def \NN {\mathrm{N}}

\def \K {\mathcal{B}}
\def \A {E}
\newcommand{\argmin}[1]{\underset{#1}{\operatorname{arg}\,\operatorname{min}}\;}

\definecolor{Colorrouge}{rgb}{0.8,0.0,0.2}
\definecolor{ColorSubSection}{rgb}{0.0,0.7,0.3}
\definecolor{ColorResume}{rgb}{0.2,0.0,0.8}
\definecolor{ColorFond}{rgb}{0.85,0.85,1}
\definecolor{Colorbleu}{rgb}{0,0,1}
\definecolor{Colorcyan}{rgb}{0,0.75,0.75}
\definecolor{Colormag}{rgb}{0.75,0,0.75}
\definecolor{bleutitre}{rgb}{0.2,0.0,0.8}

\def \virg {\, , \,\,}
\def \dsp {\displaystyle}
\def \vsp {\vspace{6pt}}

\def\sqw{\hbox{\rlap{\leavevmode\raise.3ex\hbox{$\sqcap$}}$%
\sqcup$}}
\def\findem{\ifmmode\sqw\else{\ifhmode\unskip\fi\nobreak\hfil
\penalty50\hskip1em\null\nobreak\hfil\sqw
\parfillskip=0pt\finalhyphendemerits=0\endgraf}\fi}

\begin{document}

\title{Sweeping process by prox-regular sets in Riemannian Hilbert manifolds}

\day=26 \month=01 \year=2015
\date{\today}

\author{Fr\'ed\'eric Bernicot}
\address{CNRS - Universit\'e de Nantes \\ Laboratoire de Math\'ematiques Jean Leray \\ 2, Rue de la Houssini\`ere 44322 Nantes Cedex 03, France}
\email{frederic.bernicot@univ-nantes.fr} \urladdr{http://www.math.sciences.univ-nantes.fr/$\sim$bernicot/}

\author{Juliette Venel}
\address{LAMAV \\ Universit\'e de Valenciennes et du Hainaut-Cambr\'esis Campus \\ Mont Houy 59313 Valenciennes Cedex 9, France}
\email{juliette.venel@univ-valenciennes.fr} \urladdr{http://www.math.u-psud.fr/$\sim$venel/}

\subjclass{34A60 ; 49J52 ; 58C06 ; 58C20.}
\keywords{Differential inclusion ; Sweeping process ; Prox-regularity ; Riemannian manifold.}
\thanks{F. Bernicot's research is partly supported by ANR project AFoMEN no. 2011-JS01-001-01.}

\begin{abstract}  
In this paper, we deal with sweeping processes on (possibly infinite-dimensional) Riemannian Hilbert manifolds. We extend the useful notions (proximal normal cone, prox-regularity) already defined in the setting of a Hilbert space to the framework of such manifolds. Especially we introduce the concept of local prox-regularity of a closed subset in accordance with the geometrical features of the ambient manifold and we check that this regularity implies a property of hypomonotonicity for the proximal normal cone. Moreover we show that the metric projection onto a locally prox-regular set is single-valued in its neighborhood. Then under some assumptions, we prove the well-posedness of perturbed sweeping processes by locally prox-regular sets. 
\end{abstract}

\maketitle

\section{Introduction}

\subsection{A brief review of results on sweeping processes}

The aim of this paper is to extend the study of so-called {\it sweeping process} in the setting of a (possibly infinite dimensional) Riemannian manifold.

Sweeping processes are specific differential inclusions of first order involving proximal normal cones to a moving set. They were extensively studied for the last years in the Euclidean space and more generally in a Hilbert space.

More precisely, consider a Hilbert space $H$ and a moving set $t\to C(t)$ on a time-interval
$I:=[0,T]$ and assume that the set-valued map $C:I \rightrightarrows H$ takes nonempty closed values. A function $u:I\to H$ is a solution of a problem of perturbed sweeping process if it satisfies the following differential inclusion:
\begin{equation} 
\left\{
\begin{array}{l}
 \dsp \frac{du(t)}{dt} + \NN(C(t),u(t)) - F(t,u(t)) \ni 0 \vsp \\
 u(t)\in C(t) \vsp \\
 u(0)=u_0\ ,
\end{array}
\right. \label{sys1}
\end{equation}
with an initial data $u_0\in C(0)$ and  $F: I \times H \rightrightarrows H$ a set-valued map taking nonempty compact values. Here, $\NN(C(t),u(t))$ stands for the proximal normal cone to $C(t)$ at the point $u(t)$. This differential inclusion can be thought as following: the point $u(t)$, submitted to the field $F(t,u(t))$, has to live in the set $C(t)$ and so follows its time-evolution.

\gb Let us first give a brief overview of the study for such problems. The sweeping processes have been introduced by
 J.J. Moreau in 70's  (see \cite{Moreausweep}). He considered the following problem: a point $u(t)$ has to be inside a moving convex set $C(t)$ included in a Hilbert space. When this point is catched-up by the boundary of $C(t)$, it moves in the opposite of the outward normal direction of the boundary, as if it was pushed by the physical boundary in order to stay inside the convex set $C(t)$. Then the position $u(t)$ of this point is described by the following differential inclusion
\begin{equation}
\label{Moreq}
 -\dot{u}(t) \in \partial I_{C(t)}(u(t)).
\end{equation}
Here we write $\partial I_{C}$ for the subdifferential of the characteristic function of a convex set $C$. In this work, the sets $C(t)$ are assumed to be convex and so $\partial I_{C(t)}$ is a maximal monotone operator depending on time.
To solve this problem, J.J. Moreau brings a new important idea in proposing a \textit{catching-up} algorithm.
To prove the existence of solutions, he builds discretized solutions in dividing the time interval $I$ into sub-intervals where the convex set $C$ has a little variation. Then by compactness arguments, he shows that a limit mapping can be constructed (when the length of subintervals tends to $0$) which satisfies the desired differential inclusion. \\
Thus it was the first result
concerning sweeping process (with no perturbation $F=0$) because $\partial I_{C(t)}(x) = \NN(C(t),x)$ in case of convex sets $C(t)$.

\gb Since then, important improvements have been developed by weakening the assumptions in order to
obtain the most general result of existence for sweeping
process. There are several directions: one can want to add a perturbation
$F$ as written in (\ref{sys1}), one may require a weaker assumption than the convexity of the sets $C(t)$, one would like to obtain results in Banach spaces (and not only in Hilbert spaces).

\gb In \cite{V}, M.~Valadier dealt with sweeping process by sets $C(t)=\R^n\setminus \textrm{int}(K(t))$ where $K(t)$ are closed and convex sets.
Then in \cite{CXDV}, C.~Castaing, T.X. D\'uc H\={a} and M. Valadier have studied the perturbed problem in finite dimension with convex sets $C(t)$ (or complements of convex sets). In this framework, they proved existence of solutions for (\ref{sys1}) with a convex compact valued perturbation $F$ and a Lipschitzean multifunction $C$. 
Later in \cite{Castaingper}, C.~Castaing and M.D.P.~Monteiro Marques have considered similar problems in assuming upper semicontinuity for $F$ and a ``linear compact growth'':
\begin{equation}
 F(t,x)\subset \beta(t) (1+|x|)\overline{B(0,1)} \virg \forall (t,x) \in I \times \R^n.
\label{hypF}
\end{equation}
Moreover the set-valued map $C$ was supposed to be Hausdorff continuous and satisfying an ``interior ball condition'': 
\begin{equation} 
 \exists r>0 \virg  B(0,r) \subset C(t) \virg \forall t \in I.
\label{hypC}
\end{equation}

\gb  Then the main concept, which appeared to get around the convexity of sets $C(t)$, is the notion of  ``uniform prox-regularity''. This property is very well-adapted to the resolution of (\ref{sys1}): a set $C$ is said to be {\it $\eta$-prox-regular} if the projection on $C$ is single-valued and continuous at any point whose the distance to $C$ is smaller than $\eta$. \\
Numerous works have been devoted to study sweeping processes under this assumption of prox-regularity. The case without perturbation ($F={0}$) was firstly treated by G.~Colombo, V.V.~Goncharov in \cite{Colombo}, by H.~Benabdellah in \cite{Benab} and later by L.~Thibault in \cite{Thibsweep} and by G. Colombo, M.D.P.~Monteiro Marques in \cite{Monteiro}.
In \cite{Thibsweep}, the considered problem is 
\begin{equation}
 \left \{
\begin{array}{l} 
-du \in \NN(C(t), u(t)) \vsp \\
u(0)=u_0\ ,
 \end{array}
\right.
\label{eq:mesdiffssm}
\end{equation} 
where $du$ is the differential measure of $u$. The existence and uniqueness of solutions of (\ref{eq:mesdiffssm}) are proved with similar assumptions as previously. \\
In an infinite dimensional Hilbert space, the perturbed problem is studied by M.~Bounkhel, J.F.~Edmond and L.~Thibault in \cite{Thibnonconv, Thibsweep, Thibrelax, Thibbv}.
For example in \cite{Thibbv}, the authors show the well-posedness of
\begin{equation}
 \left \{
\begin{array}{l} 
-du \in \NN(C(t), u(t)) + F(t,u(t)) dt \vsp \\
 u(0)=u_0\ ,
 \end{array}
\right.
\label{eq:mesdiffasm}
\end{equation} 
with  a set-valued map $C$ taking $\eta$-prox regular values (for some $\eta>0$) such that   
\begin{equation}
|d_{C(t)}(y) - d_{C(s)}(y)| \leq \mu(]s,t]) \virg \forall y\in H,\ \forall \  s, t \in I  \virg s\leq t
\label{varadon}
\end{equation}
where $\mu$ is a nonnegative measure satisfying
\begin{equation}
\sup_{s \in I} \mu(\{s\}) <\frac{\eta}{2}.
 \label{charge_singleton}
\end{equation}
The proof uses the algorithm developed by J.J.~Moreau with additional arguments to deal with the perturbation $F$ and the prox-regularity assumption.

\gb Indeed the main difficulty of this problem is the weak smoothness of the proximal normal cone. For a fixed closed subset $C$, the set-valued map $x\rightarrow \NN(C,x)$ is not in general upper semicontinuous and the uniform prox-regularity assumption permits us to obtain this required smoothness. To finish we mention that the sweeping processes are also studied in the framework of Banach spaces (see \cite{bena2} and \cite{BV}).

\subsection{Sweeping processes on Riemannian manifolds}

As previously explained,  sweeping processes have been extensively studied in the context of a linear ambient space (Euclidean space, Hilbert space, Banach space). However, the proof relies on the \textit{catching-up} algorithm (which only requires a projection operator) and the notion of uniform prox-regularity (which is equivalent of an hypomonotonicity property for the proximal normal cone). These observations make us want to extend the study of such differential inclusions in the framework of Riemannian manifolds, where the metric allows us to define a projection and where the Hilbertian structure of the tangent spaces is necessary to have a kind of hypomonotonicity property.

In the Hilbertian situation, the finite dimension is not useful and so we aim to work on a (possibly infinite dimensional) Riemannian Hilbert manifold (see Definition \ref{def:manifold} for details), which is locally isomorphic to a separable Hilbert space. Let $M$ be such a manifold, we consider a sweeping process on $M$, given by a set-valued map $C(\cdot)$ and a single-valued perturbation $f$
\begin{equation} 
\left\{
\begin{array}{l}
 \dsp \frac{dx(t)}{dt} + \NN(C(t),x(t)) - f(t,x(t)) \ni 0 \vsp \\
 x(t)\in C(t) \vsp \\
 x(0)=x_0\ ,
\end{array}
\right. \label{sys2}
\end{equation}
with an initial data $x_0\in C(0)$.
Note that the first differential inclusion takes place in the fiber-tangent space $T_{x(t)}M$. 
To study the well-posedness of this problem, we will generalize the notion of a prox-regular set (Definition \ref{def:prox}) which will be shown to be equivalent to a kind of hypomonotonicity property (see Theorem \ref{thm:hypo}). Some difficulties arise from the framework of manifolds.  Indeed in a Hilbertian space $H$, a nonempty closed set $C \subset H $ is said $\eta$-prox-regular or prox-regular with constant $\eta$ if for every point $z$ satisfying $d_C(z)= \underset  {c \in C } \inf  d(z,c)  < \eta$ there exists a unique point $c_z \in C$ such that $d(c_z, z)= d_C(z)$ and the map $z \mapsto c_z$ is continuous. However in infinite-dimensional Riemannian manifolds, geodesics minimizing the distance between two given points do not generally exist. Thus we will make some assumptions so that they do always exist locally. That's why we will give and use a definition of a locally prox-regular set which will take into account the geometry of the ambient manifold even if a definition of uniformly prox-regular sets will be given in Remark \ref{rem:unifprox} in specific manifolds. 

In addition we will recover an important property of prox-regular sets. Indeed we will prove that the metric projection onto a locally prox-regular set is single-valued in its neighborhood.

The notions of proximal normal cones and prox-regularity at a point are already introduced in \cite{HP} in the setting of a Riemannian manifold. However we would like to propose another definition of the proximal normal cone (related to the metric projection and better adapted to deal with problems of sweeping processes). Obviously we will show their equivalence. Moreover the prox-regularity at a point is not sufficient to treat our differential inclusions : we need a "larger'' concept of prox-regularity which we will call local prox-regularity as already specified.  
Once again these different notions are compatible. Indeed we will prove that a locally prox-regular set in our sense is prox-regular at any point in the sense of \cite{HP}. \\

Finally we prove the main result:

\begin{thm} \label{thm:general} Under some assumptions about the Riemannian Hilbert manifold $M$ (Assumption (\ref{ass:ic})), we consider a Lipschitz bounded mapping $f$ and a Lipschitz set-valued map $C(\cdot)$ taking nonempty and locally prox-regular values (Assumption \ref{ass:C}). Then for all $x_0\in C(0)$, the system (\ref{sys2}) has a unique solution $x\in W^{1,\infty}(I,M)$. 
\end{thm}
 
To our knowledge, sweeping processes on (finite- or infinite-dimensional) manifolds have not been previously studied. However other differential inclusions on Riemannian manifolds have already been treated. For example, in \cite{Grammel}, the author is interested in controllability properties of autonomous  differential inclusions: $$ \dot{x}(t) \in F(x(t))$$ in compact finite-dimensional Riemannian manifolds. The map $F$ is supposed to be Lipschitz continuous and takes nonempty convex and compact values. 
Then in \cite{Doma}, the authors prove that there exists at least one solution of the problem:
$$\dot{x}(t) \in F(t, x(t)),$$ 
with an initial condition in a  finite-dimensional Riemannian manifold. In this work, $F$ is an integrable bounded field, taking nonempty and closed values. In addition this map is supposed to be lower semicontinuous or to satisfy the Carath\'eodory conditions. Note that compactness arguments are used in their proof so the finite-dimensional assumption is crucial and by such a way, it is not possible to get uniqueness results. 

\subsection{Motivation from unilaterally constrained problems}

Let us briefly describe one of our motivations, coming from problems involving unilateral constraints.

Even in the Euclidean space $\R^d$, several applications for modelling are concerned with sweeping processes involving an admissible set 
$$ C(t):=\left\{x\in\R^d,\ \forall i\in\I,\  g_i(t,x)\geq 0 \right\}, $$
 where $(g_i)_{i\in\I}$ is a finite collection of inequality constraints.
 
To apply the existence results to the sweeping process with this moving set, we have to check that $C(t)$ is uniformly prox-regular (with a time-independent constant $\eta$), which may be very difficult. However a criterion was given in \cite[Proposition 2.9]{Juliette} and was used in several recent works (\cite{BV2,BV3,BL}). It would require to check
a kind of Mangasarian-Fromowitz constraint qualification (MFCQ) of the active gradients (we recall that the active gradients correspond to mappings $g_j$ satisfying $g_j(t,x)=0$). Let us specify that this condition describes a positive-linearly independence of the active gradients.
Of course, if the admissible set is also given by equality constraints, 
$$ C(t):=\left\{x\in\R^d,\ \forall i\in \I_1 \ g_i(t,x)=0 \quad \textrm{and} \quad \forall i\in\I_2,\  g_i(t,x)\geq 0 \right\}, $$ the same criterion cannot be verified by turning equality constraints into two  inequality ones (that is to say by writing for $i\in\I_1$, $g_i(t,x)=0$ as $g_i(t,x)\geq 0$ and $-g_i(t,x)\geq 0$). 
Indeed a gradient and its opposite are clearly not positively independent.

So it is important to have another approach for dealing with such admissible sets (involving inequalities and equalities constraints). In the case where the equality constraints are time-independent, the new point of view is to consider the manifold given by the equality constraints
$$ M:=\{x\in\R^d,\ \forall i\in\I_1,\ g_i(x)=0\}, $$ and the admissible set becomes $\tilde{C}(t):=\{x\in M,\ \forall i\in\I_2, g_i(t,x)\geq 0\}$.

Under some conditions, it is doable to check that the manifold $M$ is a smooth submanifold of $\R^d$ and so a finite dimensional Riemannian manifold. Consequently it suffices to check that the admissible set $\tilde{C}(\cdot)$ that is only defined with inequalities, is locally prox-regular.


\bigskip

The paper is organized as follows. Section \ref{sec:pre} is first devoted to some definitions and reminders about Riemannian geometry. Then in Section \ref{sec:pre2}, we extend some notions of convex analysis in the setting of a Riemannian Hilbert manifold (proximal normal cone, prox-regular set, ...) and we focus on establishing a characterisation of the prox-regularity property in terms of hypomonotonicity for the proximal normal cone (see Theorem \ref{thm:hypo}). Then we prove Theorem \ref{thm:general} in Section \ref{sec:sweep}, using the catching-up algorithm and studying the convergence of discretized solutions. 


\section{Preliminaries and definitions, related to Hilbert manifolds} \label{sec:pre}

For a complete survey on Riemannian geometry and the extension to infinite dimensional framework, we refer the reader to \cite{Klingenberg, Lang}. Let us just recall the main concepts.

\begin{df}[Hilbert manifold] \label{def:manifold} 
$M$ is a Hilbert manifold if there exists a separable Hilbert space $H$ and a smooth atlas, defining local charts from $M$ to $H$.
Then it follows that for all $x\in M$, the tangent space $T_xM$ has a Hilbert structure. Its inner product is denoted by $\langle \, ,\, \rangle$, $\langle \, , \,\rangle_x$ or  $\langle \, , \,\rangle_{T_xM}$ and the associated norm is represented by  $|\  | $ or $|\ |_{T_xM}$. 
Such a manifold is called a Riemannian Hilbert manifold if it is endowed with the following Riemannian metric: for two different points $x,y\in M$
$$ d(x,y) :=\inf \left\{ \int_a^b |\dot{\gamma}(s) |ds : \gamma   \textmd{ is a $C^1$-curve joining } x = \gamma(a)  \textmd{ and } y= \gamma(b)\right\}. $$
The bundle tangent space is denoted by $TM$ and we recall that $$TM = \underset { x \in M}  \bigcup \{  x\} \times T_x M  .$$
We denote by $\nabla^\star$ the Levi-Civit\'a connection that is the unique connection consistent with the metric $d$. 
\end{df}

In the sequel, $M$ will be a connected Riemannian Hilbert manifold and $\nabla^*$ its natural Levi-Civita connection.
Using this connection, we can define the parallel transport:

\begin{df}[Parallel transport along a curve]  For a $C^1$-curve $\gamma :  [0, T] \rightarrow M$ and $v \in T_{\gamma(0)}M$, the parallel transport of $v$ along $\gamma$ is the unique vector field $t \mapsto V(t)$ satisfying 
$$ V(t) \in  T_{\gamma(t)}M \textmd{  and  } \left\{
\begin{array}{l}
 {\nabla^\star}_{\dot \gamma(t)} V(t)=0 \\
 V(0)=v.
\end{array}  \right. $$
\end{df}

\begin{df}[Geodesics] The geodesics are defined to be (locally) the shortest path between two points. As a consequence, a geodesic $\gamma$ is also a curve such that parallel transport along it preserves the tangent vector to the curve, so
  $$  \nabla^\star_{\dot\gamma} \dot\gamma= 0.$$
\end{df}

For any point $x\in M$ and for any vector $v\in T_xM$, there exists a unique geodesic $\gamma_{x,v} \, : I \rightarrow M$ such that
$$ \left\{\begin{array}{ll} 
       \gamma_{x,v}(0) & = x \\
      \dot\gamma_{x,v}(0) & = v,    
          \end{array}
       \right. $$
where $I$ is a maximal open interval in $\R$ containing $0$ (depending on $x$ and $v$). 

\begin{df} The Riemannian Hilbert manifold $M$ is said geodesically complete if
   every maximal geodesic is defined on $\R$ (i.e. $I=\R$). 
\end{df}

From now on, we assume that the connected Riemannian Hilbert manifold $M$ is geodesically complete.

\begin{df}[Exponential map] 
 If $M$ is geodesically complete, we recall that the exponential map at $x$, denoted by $\exp_x$, is defined in the whole tangent space $T_xM$ and maps each vector $v$ of the tangent space
 $T_xM$ to the end of the geodesic segment $[x,y]$ in $M$ issuing from
 $x$ in the direction $v$ and having length $d(x,y) = |v |$.
\label{def:expmap}
\end{df}

\begin{df}[Injectivity radius]  We define the injectivity radius of
  $M$ at $x$, $i_M(x)$ as the supremum of the numbers $r>0$ of all balls
  $B(0_x,r)$ in $T_xM$ for which $\exp_x$ is a diffeomorphism from
  $B(0_x,r)$ onto $B(x,r)$. Then, we denote the global injectivity
  radius of $M$ by
$$ i(M):= \inf_{x\in M}  i_M(x).$$
\end{df}

\begin{df}[Convexity radius] The convexity radius $c_M(x)$ of a point $x\in M$ is the supremum of the numbers $r>0$ such that the ball $B(x,r)$ is convex (in the sense that given two points $p,q\in B(x,r)$, there exists a unique geodesic in $B(x,r)$ joining them and of length $d(p,q)$). 
 We denote the global convexity radius by
 $$ c(M):= \inf_{x\in M} c_M(x).$$
\end{df}

\noindent Similarly we define for a bounded subset $\mathcal{B}$ of $M$, $$i(\mathcal{B}):= \inf_{x\in \mathcal{B}}  i_M(x) \quad \textrm{and} \quad c(\mathcal{B}) : = \inf_{x\in\mathcal{B} } c_M(x).$$

\noindent Note that for $x \in \mathcal{B}$, the open ball $B(x, c(\mathcal{B})) $ is not in general included in $\BB$. \\

\noindent Let $\BB$  a bounded subset of $M$ and $x$, $y  \in \BB$ satisfying $d(x,y) < c(\BB) $,  we define $\Gamma_{x,y}\in T_xM$ as the unique vector $w\in T_xM$ satisfying 
$$ \gamma_{x,w}(1)=y.$$
 We can note that its norm  is equal to the distance between
$x$ and $y$: \be{normgam} |\Gamma_{x,y}| = d(x,y).\ee
According to Definition \ref{def:expmap}, we have $y= \exp_x(\Gamma_{x,y})$. In other words, the vector $\Gamma_{x,y}$ could be also expressed
as follows:  $$ \Gamma_{x,y} =
\exp_x^{-1}(y).$$

\begin{df}[Sectional curvature]  Let $x \in M $ and $\sigma_x \subset T_{x}M$ a two-dimensional plane, the sectional curvature $K(\sigma_x)$  is the Gaussian curvature of the section of $M$ whose tangent space at $x$ is $\sigma_x$
(in other words, it is the Gaussian curvature of the range of $\sigma_x$ under the exponential map at $x$).  
\end{df}

Since the interesting notions have been recalled, we can now specify the hypotheses about the manifold.

\begin{ass} \label{ass:ic}
We assume that $M$ is a connected Riemannian Hilbert manifold which is metrically complete and geodesically complete and $M$ has a locally bounded geometry : for every bounded subset $\mathcal{B} \subset M$, 
\begin{itemize}
 \item $i(\mathcal{B} ) >0$ and  $c(\mathcal{B})>0$
 \item the sectional curvatures are bounded in $\mathcal{B} $:  $ |K(\mathcal{B})| < + \infty $
$$ K(\mathcal{B}) : = \underset{ \begin{array}{c}x\in \mathcal{B} \\ \sigma_x \subset T_xM \end{array}}{\sup}K( \sigma_x) $$
\end{itemize}
Thus we define for a bounded subset $\mathcal{B} \subset M$,
 \be{rho0}  \rho(\mathcal{B}):= \min\{i(\mathcal{B}),c(\mathcal{B}), \pi/(2\sqrt{|K(\mathcal{B})|})\}. \ee
Moreover the exponential map $\exp$ is supposed to be locally smooth on $TM$ and the inverse map $D\exp^{-1}_x:TM \rightarrow T_xM $ is supposed to be locally Lipschitz: for every bounded subset $\mathcal{B} \subset M$ and for $k \in\{0, 1,2\}$,  there exist $C_e({\mathcal B }) >0 $ and $L(\mathcal{B}) >0 $ such that
\begin{equation} \|\exp\|_{C^k({\mathcal B})} \leq L(\mathcal{B})  \label{eq:exp} \end{equation}
and for all
$  x,y,z\in{\mathcal B}$ satisfying $ d(x, y) < \rho(\BB)$ and  $ d(x, z) < \rho(\BB)$,  for all $v\in T_yM,  w\in T_zM  $ 
$$
\begin{array}{l}
| \exp^{-1}_x (y) -\exp^{-1}_x(z) |_{T_xM} \leq C_e({\mathcal B })d(y,z),  \vsp \\
  | D\exp^{-1}_x(y)[v] -D\exp^{-1}_x(z)[w] |_{T_xM} \leq C_e({\mathcal B })d_{TM}((y,v),(z,w)),
  \end{array}
  $$
where $d_{TM}$ represents the distance on the tangent bundle: 
 $$ d_{TM}((y,v),(z,w)) := d(y,z)+ |L_{y\rightarrow z} v - w|_{T_zM}$$
and $L_{y\rightarrow z}$ is the tangential parallel transport from $y$ to $z$ (see Definition \ref{def:partrans}).  
\end{ass}

\medskip

\begin{rem}
In the finite dimensional case, it suffices to assume that the connected Riemannian manifold $M$ is geodesically complete because in that case Hopf-Rinow's Theorem asserts that $M$ is automatically metrically complete and  the closed and bounded subsets of $M$ are compact. Since the functions $i_M$, $c_M$ are continuous, the quantities $i(\mathcal{B})$ and  $c(\mathcal{B})$ are necessarily strictly positive. In the same way, $|K(\mathcal{B})| < +\infty $ is also satisfied.   
\end{rem}


We aim to finish this section with the following definition:

\begin{df}[Local parallel transport]
Let $\mathcal{B} \subset M$ be a bounded subset and let $x,y\in \mathcal{B} $ two different points satisfying $d(x,y) \leq\rho (\mathcal{B}) \leq c(\BB)$, then there exists a unique geodesic curve $\gamma$ with $\gamma(0)=x$ and $\gamma(1)=y$ and so $ \Gamma_{x,y} $ is well-defined. 
 Moreover, we denote by $L_{x\rightarrow y}:T_xM \rightarrow T_yM$ the tangential parallel transport defined as follows: for all $v\in T_xM$, $L_{x\rightarrow y}(v)=V(1)$ where $V$ is the unique vector field satisfying $V(0)=v$ and $ {\nabla^\star}_{\dot \gamma(t)} V(t)=0$. The map $L_{x\rightarrow y}$ also defines a linear isometry between $T_xM$ and $T_yM$ with $L_{x\rightarrow y}^{-1}=L_{y\rightarrow x}$.
 \noindent Moreover it satisfies
\be{eq:symmetry} L_{x\rightarrow y}\left(\Gamma_{x,y}\right) = -\Gamma_{y,x}. \ee
\label{def:partrans}
\end{df}


\section{Technical preliminaries about convex analysis and the squared distance} \label{sec:pre2}

In this section, we aim to extend the notion of {\it prox-regular subset} to the setting of Riemannian Hilbert manifold and to understand how the main property of {\it hypomonotonicity} (of the proximal normal cone) is modified (see Theorem \ref{thm:hypo}).

In all the sequel we consider a manifold $M$ which satisfies Assumption \ref{ass:ic}. 
We first recall some properties about the smoothness of the squared distance:

\begin{prop}[Proposition 2.2 \cite{AF}] For every bounded subset $\mathcal{B}$ of $M$, there exists $C_{\rho(\mathcal{B})} >0$ such that  
for all $p,q \in \mathcal{B} $ satisfying $d(p,q)\leq \rho(\mathcal{B}) $:
\be{hyp:H} \|H_ x d^2(p,q) \|_p \leq C_{\rho(\mathcal{B})}, \ee
where $$H_ x d^2(p,q) := D^2 \phi(p)$$ with $$\phi(p):= d^2(p,q) \textmd{ and } \| \Psi \|_p :=  \underset {v \in T_pM, |v|=1} {\sup} |\Psi [v,v]|.$$
\label{propH}
\end{prop}

\begin{lem} \label{lem:geodesic} For all bounded set $ \K$ included in
  $M$, for all $ \tau \in [0,1]$ and $x,y\in \K$ satisfying $d(x,y) < \rho(\mathcal{B}) $, we have 
\be{hyp:geod} |\ddot{\gamma}_{x,\Gamma_{x,y} }(\tau) | \leq L(\mathcal{B})
|\Gamma_{x,y}|^2 = L(\mathcal{B})  d(x,y)^2,   \ee
where $ {\gamma}_{x,\Gamma_{x,y} }(\tau)  =   \exp_x(\tau \Gamma_{x,y}) =  \exp_x(\tau \, \exp_x^{-1} ( y))$. 
\end{lem}

\begin{proof} Let $x,y\in \K$ and consider for $\tau \in[0,1]$
$$ \gamma_{x,\Gamma_{x,y}}(\tau) = \exp_x(\tau \Gamma_{x,y}).$$
By differentiating in $\tau$, we get
$$ \ddot \gamma_{x,\Gamma_{x,y}}(\tau)  = D^2 \exp_x (\tau \Gamma_{x,y}) [\Gamma_{x,y},\Gamma_{x,y}].$$ 
Then the assumption \eqref{eq:exp} gives the expected result. 
\end{proof}

As a particular case of \cite[Theorem 1.2]{Karcher} (with a constant function $f$), we recall the link between the exponential map and the gradient of the squared distance.

\begin{lem} \label{lemb} For every bounded set $\BB \subset M $, for all $x,y\in \BB$ satisfying $d(x,y) < \rho(\BB)$, we have
\begin{equation}
  \label{eq:gamma} \nabla_x d^2(x,y) = 2 d(x,y) \nabla_x d(x,y)= -2 \Gamma_{x,y}=-2 \exp_x^{-1}(y) . 
  \end{equation}
\end{lem}
\noindent For the sake of readability, we give a short proof.

\begin{proof} By definition of $\rho(\BB)$, there exists a unique geodesic $\gamma_{x,\Gamma_{x,y}}$ such that $\gamma_{x,\Gamma_{x,y}}(1)=y$. For all $t\in[0,1)$, the constant velocity along the geodesic implies that
$$ d(\gamma_{x,\Gamma_{x,y}}(t),y) = d(x,y) -t|\Gamma_{x,y}| = (1-t) d(x,y).$$
By differentiating at $t=0$, we get
$$\langle \nabla_x d(\gamma_{x,\Gamma_{x,y}}(0),y), \dot \gamma_{x,\Gamma_{x,y}}(0) \rangle = \langle \nabla_x d(x,y), \Gamma_{x,y}\rangle = -d(x,y).$$
From $|\nabla_x d(x,y)|=1$ and $|\Gamma_{x,y}|=d(x,y)$, we deduce (\ref{eq:gamma}).
 \end{proof}

\begin{rem}
With the notations of Proposition \ref{propH}, $\nabla_x d^2(p,q)= \nabla \phi (p)$. We deduce from this proposition and from (\ref{eq:gamma}) that the map $p \mapsto \exp_p^{-1}(q) $ is smooth on $\BB$. 
\label{rem:smoothexp}
\end{rem}

\begin{rem}
In the same way,  $$\nabla_y d^2(x,y) = 2 d(x,y) \nabla_y d(x,y)= -2 \Gamma_{y,x}=-2 \exp_y^{-1}(x).$$
\label{rem:nablay}
\end{rem}

\begin{df} Let $C$ be a closed subset of $M$ and $x\in C$. The proximal normal cone to $C$ at $x$ is defined as follows:
$$ \NN(C,x):=\left\{ v\in T_xM,\ \exists \varepsilon>0,\ x\in \PPP_C(\gamma_{x,v}(\varepsilon))\right\},$$
 where $\PPP_C$ is the projection onto $C$
 $$ \PPP_C(y):= \left\{z\in C,\ d(y,z)=d_C(y)=\inf_{c\in C} d(y,c) \right\}.$$
 \label{def:cone}
\end{df}

\noindent Note that $\NN(C,x) $ is a cone because for all $\alpha >0$, for all $
v \in \NN(C,x)  $, $\alpha v  $ belongs to $\NN(C,x) $ due to the
equality $\gamma_{x,v}(\varepsilon) = \gamma_{x,\alpha v}\left(\frac{\varepsilon}{\alpha}\right)  $. 

\begin{df} Consider a nonempty closed subset $C\subset M$, $ C$ is said locally prox-regular if for all bounded set $\mathcal{B} \subset M $, there exists $ 0 <\eta_\BB \leq \rho(\BB)$ such that for every  $t\in[0,\eta_\BB)$, $x\in C \cap \BB$ and $v\in \NN(C,x) \setminus \{0\}$
$$ x\in \PPP_C\left(\gamma_{x,\frac{v}{|v|}}(t) \right).$$
\label{def:prox}
\end{df}

\begin{rem}
Note that we can define a uniformly prox-regular set : a closed subset $ C\subset M $ will be called $\eta
  $-prox-regular or uniformly prox-regular with constant $\eta$ if for all $t\in[0,\eta)$, $x\in C$ and $v\in \NN(C,x) \setminus \{0\}$, 
$$ x\in \PPP_C\left(\gamma_{x,\frac{v}{|v|}}(t) \right).$$
However this definition is meaningful if $ i(M) > 0$, $c(M)> 0 $ and $|K(M)| < + \infty$ (for example il the framework of compact manifolds). 
\label{rem:unifprox}
\end{rem}

\noindent Now we want to prove that the local prox-regularity provides a kind of hypomonotonicity property. This result is the purpose of the following lemma. 

\begin{lem} \label{lem:hypo} Let $C$ be a closed subset of $M$ and assume that $C$ is locally prox-regular. 
Then for all bounded subset $\K$ of
  $M$, there exists a constant $\A_\K >0$ such that for all $x,y\in C
  \cap \K$ satisfying $d(x,y) \leq \rho(\BB)$ and $v\in \NN(C,x)$
$$  \langle v,  \Gamma_{x,y} \rangle_{T_xM} \leq \A_\K |v|  d(x,y)^2.$$ 
\end{lem}
\begin{proof} Let $x,y\in C \cap \K$ and $v\in\NN(C,x)$. We can assume $|v|=1$ without loss of generality. We define $\tilde{\BB}:=\{m \in M, d(m, \BB)\leq \eta_\BB\}$  (where $\eta_\BB$ is introduced in Definition \ref{def:prox}) and we set $r:=\min(\eta_\BB ,\rho({\tilde{\BB}}) )$.
If $d(x,y)\geq r/2 $, the result is easily proved by (\ref{normgam}) with $\A_\K=2/r$. \\
Consider also $x,y$ with $d(x,y) < r/2$. Since $C$ is locally prox-regular, for all $t\in(0,r/2)$
$$ x\in \PPP_C\left(\gamma_{x, v}(t) \right)$$
which implies
$$ t^2 = d(x, \gamma_{x, v}(t))^2 \leq d(y, \gamma_{x, v}(t))^2.$$
By a second order expansion and setting $m:=\gamma_{x, v}(t)$, we get thanks to (\ref{eq:gamma})
\begin{align*}
  d(y,m)^2 & =  d(x,m)^2 +  \int_0^1 \left[ \frac{d}{d\tilde s}  d ^2 (\gamma_{x, \Gamma_{x,y}}(\tilde s),m) \right] _{|\tilde s = s}ds \\
  & =  d(x,m)^2 +  2 d(x,m)\langle \nabla_x d(x,m),\Gamma_{x,y}\rangle \\ 
  & \quad + \int_0^1 \int_0^s \left[\frac{d^2}{d\tilde s^2} d ^2 (\gamma_{x, \Gamma_{x,y}}(\tilde s),m)\right] _{|\tilde s = \tau} d\tau ds.
\end{align*}
Consequently, 
\begin{align*}
 -2 d(x,m)\langle \nabla_x d(x,m,\Gamma_{x,y}\rangle & \leq  \int_0^1 \int_0^s \left[\frac{d^2}{d\tilde s^2} d ^2 (\gamma_{x, \Gamma_{x,y}}(\tilde s),m)\right] _{|\tilde s = \tau} d\tau ds \\
 & \leq \int_0^1 \int_0^s \left[ H_x d^2(\gamma_{x, \Gamma_{x,y}}(\tau),m)[\dot \gamma_{x,\Gamma_{x,y}}(\tau),\dot \gamma_{x,\Gamma_{x,y}}(\tau)] \right. \\
& \qquad \qquad \left. + \langle \nabla_x d^2(\gamma_{x,
    \Gamma_{x,y}}(\tau),m), \ddot \gamma_{x,\Gamma_{x,y}}(\tau)
  \rangle  \right] d\tau ds.
\end{align*} 
Since $d(x,\gamma_{x,\Gamma_{x,y}}(\tau)) \leq  d(x,y) < r/2 $ and $d(x, \gamma_{x,v}(t)) \leq t\leq r/2$, it is clear that $\gamma_{x,\Gamma_{x,y}}(\tau) $ and $\gamma_{x,v}(t)$  belong to $B(x,r/2)\subset \tilde{\BB} $.
 Moreover\be{eq:di} d(\gamma_{x, \Gamma_{x,y}}(\tau),m)=
d(\gamma_{x, \Gamma_{x,y}}(\tau),\gamma_{x,v}(t))<r/2+ r/2< r  \leq \rho(\tilde{\BB}). \ee Thus by
Proposition \ref{propH}, $$| H_x d^2(\gamma_{x, \Gamma_{x,y}}(\tau),m)[\dot
\gamma_{x,\Gamma_{x,y}}(\tau),\dot \gamma_{x,\Gamma_{x,y}}(\tau)] |\leq
C_{\rho(\tilde{\BB})} |\Gamma_{x,y}|^2 .$$
 Furthermore Lemma \ref{lem:geodesic} with (\ref{eq:di}) gives  $$\langle \nabla_x d^2(\gamma_{x,
    \Gamma_{x,y}}(\tau),m), \ddot \gamma_{x,\Gamma_{x,y}}(\tau)
  \rangle \leq 2 d(\gamma_{x,
    \Gamma_{x,y}}(\tau),m) L(\tilde{\BB}) |\Gamma_{x,y}|^2\leq 2 r L(\tilde{\BB})
  |\Gamma_{x,y}|^2.  $$
Thus by adding the previous inequalities we obtain 
\begin{align*}
 -2 d(x,m)\langle \nabla_x d(x,m),\Gamma_{x,y}\rangle & \leq C_{\rho(\tilde{\BB})}
 |\Gamma_{x,y}|^2 + 2 r L(\tilde{\BB})  |\Gamma_{x,y}|^2.
\end{align*}
So we deduce from Lemma \ref{lemb}  and (\ref{normgam}) that
$$  \langle \exp_x^{-1}( \gamma_{x,v}(t) ), \Gamma_{x,y}\rangle \leq
\frac{C_{\rho(\tilde{\BB})}+2 r L(\tilde{\BB})}{2} d(x,y)^2.$$
As $\exp_x^{-1}( \gamma_{x,v}(t) ) = t v $, it comes 
$$  \langle v, \Gamma_{x,y}\rangle \leq
\frac{C_{\rho(\tilde{\BB})}+2 r L(\tilde{\BB}) }{ 2t} d(x,y)^2.$$
This inequality holds for all $t \in (0 , \ r/2)$ which ends the proof by
setting $\A_\K:=  \max(C_{\rho(\tilde{\BB})}/r+2 L(\tilde{\BB}), 2/r)  .$

\end{proof}

\begin{rem}
Note that the constants $E_\BB$ depend on the geometrical features of $M$ but also on $C$ through the constant $\eta_\BB$.   
\label{rem:constantE}
\end{rem}

\begin{rem} \label{rem:hypo} Let $C$ be a closed subset of $M$ and $x \in C$. If $v \in \NN(C,x)$ then there exists $\alpha>0$ and $\Lambda>0$ such that for all $y\in C
  \cap B(x,\alpha)$
$$  \langle v, \Gamma_{x,y} \rangle_{T_xM} \leq \Lambda |v | d(x,y)^2.$$
Indeed, assume that $|v|= 1$ (without loss of generality) since $v \in \NN(C,x)$, there exists $\varepsilon >0$, such that  $x \in \PPP_C(\gamma_{x, v}(\varepsilon))$ by Definition \ref{def:cone}. Now it suffices to fix $\BB= B(x, \varepsilon)$ and take $\alpha := \min(\varepsilon, \rho(\BB)/2)$. Thus for all $t \leq \alpha $, $ x \in \PPP_C(\gamma_{x, v}(t)) $. By following the previous proof, we can choose $\Lambda:= \frac{C_{\rho(\BB)}}{2 \alpha} + 2 L(\BB) $. Here $\alpha$ and $\Lambda$ depend on $v$ and may degenerate.  The important property of a locally prox-regular set is that these two quantities may be assumed to be independent on $v$, as proved in Lemma \ref{lem:hypo}. 
\end{rem}

The following statement will be useful in Lemma \ref{lem:hypo2} to check that the hypomonotonicity property (which is detailed in Lemma \ref{lem:hypo}) implies the local prox-regularity. Moreover its corollary will allow us to show that  the projection onto a locally prox-regular set is single-valued in its neighborhood.

\begin{lem}[{\cite[Theorem 1.2]{Karcher}}] \label{lem:curvature}  Consider a manifold $M$  satisfying Assumption \ref{ass:ic} and $\mathcal{B}$ a bounded subset of $M$. 
 \begin{itemize}
 \item If $K(\mathcal{B}) \leq 0$
   then for every $x\in \BB$ and along any geodesic $\gamma$ taking
   values in $B(x, r) \subset \BB $ with $r \leq c(\mathcal{B})$,
 $$ \frac{d^2}{dt^2} d^2(\gamma(t),z) \geq 2 |\dot \gamma(t)|^2.$$

 \item If $ 0< K(\mathcal{B}) \leq \delta $,
  with $S(t):= t \cot (t)$ and \hbox{$r \in \left] 0, \rho(\mathcal{B}) \right]$}, we have for all
   $x\in \BB$ and along any geodesic $\gamma$ taking values in $B(x, r) \subset \BB $,
 $$ \frac{d^2}{dt^2} d^2(\gamma(t),z) \geq 2 S(r \delta^{1/2} ) |\dot \gamma(t)|^2.$$

\end{itemize}
\end{lem}

\begin{cor} \label{cor:new} 
Let  $\mathcal{B}$ a bounded subset of $M$ containing  $x $, $ z_1$ and $z_2$. Suppose that
\begin{itemize}
\item $z_1,z_2 \in B(x, r) \subset \BB$ with $r \leq c(\mathcal{B})$, if $K(\mathcal{B}) \leq 0$; 
\item $z_1,z_2 \in B(x, r) \subset \BB $ with $r < \rho(\mathcal{B})$ if $0 <K(\mathcal{B}) \leq \delta$ .
\end{itemize}
then there exists $A:=A(\BB) >0 $ such that
$$ \langle \Gamma_{z_2,x} - L_{z_1 \to z_2} (\Gamma_{z_1,x}) , \Gamma_{z_2,z_1} \rangle \geq A d(z_1,z_2)^2.$$ 
\end{cor}

\begin{proof} We set $\varphi(z):= d^2(z,x)$. Then with a second order expansion of the function $d^2$ and \eqref{eq:gamma}, 
 we have
$$ \varphi(z_1)-\varphi(z_2) = \langle \nabla \varphi(z_2), \exp^{-1}_{z_2}(z_1) \rangle + \int_0^1 \frac{d^2}{dt^2} \varphi(\gamma(t)) dt$$
where $\gamma(t)$ is the geodesic between $\gamma(0)=z_2$ and $\gamma(1)=z_1$.
With $A=2$ (in the case of negative curvature) or $A=  2S(r \delta^{1/2} )$ (if the curvature is at most $\delta$), Lemma \ref{lem:curvature} implies
$$ \varphi(z_1)-\varphi(z_2) \geq \langle \nabla \varphi(z_2), \exp^{-1}_{z_2}(z_1) \rangle +  A d(z_1,z_2)^2,$$
which can be written as
$$ \varphi(z_1)-\varphi(z_2) \geq -2 \langle \Gamma_{z_2,x}, \Gamma_{z_2,z_1} \rangle + A d(z_1,z_2)^2.$$
By symmetry, we have
$$ \varphi(z_2)-\varphi(z_1) \geq -2 \langle \Gamma_{z_1,x}, \Gamma_{z_1,z_2} \rangle + A d(z_1,z_2)^2.$$
By summing these two inequalities, it comes 
$$  \langle \Gamma_{z_2,x}, \Gamma_{z_2,z_1} \rangle_{T_{z_2}M} +  \langle \Gamma_{z_1,x}, \Gamma_{z_1,z_2} \rangle_{T_{z_1} M} \geq  A d(z_1,z_2)^2.$$
Using parallel transport along the geodesic from $z_1$ to $z_2$, 
$$ \langle \Gamma_{z_1,x}, \Gamma_{z_1,z_2} \rangle_{T_{z_1} M} = -\langle L_{z_1 \to z_2} \Gamma_{z_1,x}, \Gamma_{z_2,z_1} \rangle_{T_{z_2} M},$$
which is  the expected result.
\end{proof}

\begin{lem} \label{lem:hypo2} Let $C$ be a closed subset of $M$. Assume that for every bounded subset $\BB$ of $M$, there exists a constant $\kappa_\BB >0$ such that for all $x,y\in C \cap \BB$ satisfying $d(x,y) < \rho(\BB)$ and every $v\in \NN(C,x)$ 
\be{eq:hypo}  \langle v, \Gamma_{x,y} \rangle_{T_xM} \leq \kappa_\BB |v| d(x,y)^2.\ee

Then there exists $\theta^\BB \in ]0, \rho(\BB)[ $ such that for all  $t <\theta^\BB $ and for all $v \in \NN(C,x) \setminus \{ 0 \}$, 
$$ x \in \PPP_C\left(\gamma_{x, \frac{v}{|v|}} (t)\right). $$
In other words, $C$ is locally prox-regular.
\end{lem}

\begin{proof} Let  $\BB$ be a bounded subset of $M$, let us fix  $x\in C\cap \BB$ and $v\in \NN(C,x)$. Without loss of generality assume that $|v|=1$. 
We define $\tilde{\BB} :=\{  m\in M, d(m,\BB) <  \rho(\BB)\}$ and we recall that $ \rho(\tilde{\BB})\leq\rho(\BB)$. 
Moreover we set $\theta^\BB := \min\left(\frac{1}{2\kappa_{\tilde{\BB} }} ,\frac{\rho(\tilde{\BB} )}{3} \right)$ if $ K(\tilde{\BB} ) \leq 0$ and $ \theta^\BB := \min\left(\frac{\rho(\tilde{\BB} )}{3}, \frac{1}{3\delta^{1/2}} \arctan(\frac{3 \delta^{1/2}}{ 2\kappa_{\tilde{\BB} }})\right) $ if $0 <K(\tilde{\BB} ) \leq \delta$.  
We are going to check that for $t < \theta^\BB< \rho(\tilde{\BB})/3$
\be{eq:distamontrer} d_C(\gamma_{x,v}(t)) = t. \ee
We assume that (\ref{eq:distamontrer}) does not hold and then there exists $z \in C$ ($z \neq x$) satisfying 
$$ d(z,\gamma_{x,v}(t))^2<t^2$$
and in using a second order expansion, we get
\begin{align*}
 t^2 > d(z,\gamma_{x,v}(t))^2=t^2+\langle \nabla_x
d^2(x,\gamma_{x,v}(t)), \Gamma_{x,z} \rangle \\ + \int_0^1 \int_0^\sigma
\left[ \frac{d^2}{ds^2}
  d^2(\gamma_{x,\Gamma_{x,z}}(s),\gamma_{x,v}(t))\right ]_{|s=\tau}
d\tau d\sigma.
\end{align*}
Note that $\Gamma_{x,z}$ is well-defined since $d(x,z) <  d(x, \gamma_{x,v}(t) ) + d(z,
\gamma_{x,v}(t) )  <2t < \rho(\tilde{\BB})$. 
As a consequence,  $z\in \tilde{\BB} $ and $$\begin{array}{l}
 \dsp \int_0^1  \int_0^\sigma
\left[\frac{d^2}{ds^2} d^2(\gamma_{x,\Gamma_{x,z}}(s),\gamma_{x,v}(t))
\right ]_{|s=\tau} d\tau d\sigma  \vspace{6pt} \\  
\hspace{1cm}  \leq \dsp \langle - \nabla_x
d^2(x,\gamma_{x,v}(t)), \Gamma_{x,z} \rangle =  \langle 2
\exp_x^{-1}(\gamma_{x,v}(t)), \Gamma_{x,z} \rangle= \langle 2 t v,
\Gamma_{x,z} \rangle, 
\end{array}$$
thanks to Lemma \ref{lemb}. Thus with assumption (\ref{eq:hypo}), it comes
$$ \int_0^1  \int_0^\sigma
\left[\frac{d^2}{ds^2} d^2(\gamma_{x,\Gamma_{x,z}}(s),\gamma_{x,v}(t))
\right ]_{|s=\tau} d\tau d\sigma \leq 2 t \kappa_{\tilde{\BB}} d(x,z)^2.$$
Moreover 
\begin{align*}
d( \gamma_{x,\Gamma_{x,z}}(s),\gamma_{x,v}(t) ) &  \leq d(x,
\gamma_{x,v}(t) ) + d(x, \gamma_{x,\Gamma_{x,z}}(s)) \leq t |v |+ s
|\Gamma_{x,z} | \\
& \leq t + d(x,z) \\ &\leq 3t \leq \rho(\tilde{\BB}) \leq c(\tilde{\BB}) .
\end{align*}
If $ K(\tilde{\BB})\leq 0$ everywhere, Lemma
\ref{lem:curvature} yields for $t \leq  \frac{\rho(\tilde{\BB})}{3} \leq \frac{c(\tilde{\BB})}{3}$
$$ \left[ \frac{d^2}{ds^2} d^2(\gamma_{x,\Gamma_{x,z}}(s),\gamma_{x,v}(t)) \right ]_{|s=\tau} \geq 2|\dot \gamma_{x,\Gamma_{x,z}}(\tau) |^2 = 2|\Gamma_{x,z}|^2 = 2d(x,z)^2.$$
Consequently, we have for $t \leq \frac{\rho(\tilde{\BB})}{3} \leq \frac{c(\tilde{\BB})}{3}$
$$ d(x,z)^2  \leq  2t \kappa_{\tilde{\BB}} d(x,z)^2$$
which is impossible for $t<\frac{1}{2\kappa_{\tilde{\BB}}}$.
Thus for $t <\theta^\BB$, $d_C(\gamma_{x,v}(t))=t  $ so $x\in \PPP_C(\gamma_{x,v}(t)) $. 
If $0 <K(\tilde{\BB}) \leq \delta$, Lemma
\ref{lem:curvature} yields for $t \in \left]0,  \frac{\rho(\tilde{\BB})}{3} \right]$
$$ \left[ \frac{d^2}{ds^2}
  d^2(\gamma_{x,\Gamma_{x,z}}(s),\gamma_{x,v}(t)) \right ]_{|s=\tau}
\geq  2 S(3t \delta^{1/2})d(x,z)^2 ,$$
for every $\tau\in [0,s]$ and every $s \in [0,1]$. 
Hence, we get 
$$ d(x,z)^2  \leq \frac{ 2\kappa_{\tilde{\BB}} t }{ S(3t \delta^{1/2})}d(x,z)^2
=\frac{2 \kappa_{\tilde{\BB}} \tan(3t \delta^{1/2})}{3 \delta^{1/2}} d(x,z)^2$$
which is impossible as soon as $t < \theta^\BB$.
\end{proof}

Combining Lemmas \ref{lem:hypo} and \ref{lem:hypo2}, we obtain the following characterization of locally prox-regular subsets:

\begin{thm}[Hypomonotonicity property] \label{thm:hypo} Let $C$ be a closed subset of $M$. Then $C$ is locally prox-regular if and only if for all bounded subset $\BB$ of
  $M$, there exists a constant $\A_\BB >0$ such that for all $x,y\in C
  \cap \BB$ with $d(x,y) \leq \rho(\BB)$ and $v\in \NN(C,x)$
$$  \langle v,  \Gamma_{x,y} \rangle_{T_xM} \leq \A_\BB |v|  d(x,y)^2.$$
\end{thm}

As mentioned in the introduction, the notions of proximal normal cone and of prox-regularity (at a point) have been already appeared in \cite{HP}. Let us recall them. 

\begin{df}[\cite{HP}]
For $C$ a closed subset of $M$ and $x\in C$ the proximal normal cone to $C$ at $x$ is defined by 
$$ N^P_C(x):= N^P_{\exp^{-1}_x (U\cap C)} (0_x)$$
with $U := \exp_x(V)$, where $V$ is an open neighborhood of $0_x$. \\
Moreover the limiting normal cone to $C$ at $x$ is expressed as follows: 
$$N^L_C(x):=\left\{ \lim_{i \to \infty} v_i : v_i \in N^P_C(x_i) \textmd{ and } x_i \to x \textmd{ with } x_i \in C\right\}. $$
\end{df}

\begin{df}\cite[Definition 3.3]{HP}
 A closed subset $C\subset M$ is said to be prox-regular at $\bar{x} \in C$ if there exist $\epsilon>0$ and $\alpha>0$ such that $B(\bar{x}, \epsilon)$ is convex  and  for every $x\in C\cap B(\bar{x}, \epsilon)$, $y\in C \cap B(\bar{x},\epsilon)$ and $v\in N^L_C(x)$ with $|v|\leq \epsilon$ we have
$$ \langle v , \exp^{-1}_x(y) \rangle \leq \frac{\alpha}{2} d(x,y)^2.$$
\label{def:proxHP}
\end{df}

 First we would like to show that the two notions of proximal normal cones are equivalent. This result is based on the following proposition.

\begin{prop} \label{prop:hypomo} Let $C$ be a closed subset of $M$ and $\BB$ a bounded subset of $M$. Consider $x\in C \cap \BB$ and $v\in T_xM$. Then $v\in \NN(C,x)$ if and only if there exist $ \theta \in (0, \rho(\BB))$ and  $\Lambda  >0$ such that for all $y\in C \cap \BB$ with $d(x,y)\leq  \theta$ 
\be{eq:hypobis}  \langle v, \Gamma_{x,y} \rangle_{T_xM} \leq \Lambda  d(x,y)^2.\ee
\end{prop}

\begin{proof} Let $v \neq 0$ satisfying (\ref{eq:hypobis}). For all $ y  \in C \cap \BB$ with $d(x,y)\leq  \theta$,  we have $$ \langle v, \Gamma_{x,y} \rangle_{T_xM} \leq \kappa |v|  d(x,y)^2 ,$$ where $\kappa:= \frac{\Lambda}{|v|}$. Then it suffices us to check that for some small enough $t>0$, 
$$ x\in \PPP_{C}(\gamma_{x,v}(t)),$$
which is equivalent to 
$$ d_C(\gamma_{x,v}(t)) = t|v|.$$
Let assume that it does not hold: $d_C(\gamma_{x,v}(t)) < t|v|$ and choose $z$ satisfying $ d(z,\gamma_{x,v}(t) ) < t|v|$. Following the proof of Lemma \ref{lem:hypo2}, it comes a contradiction for small enough $t$  ($t<\theta^\BB$). In conclusion, $v \in \NN(C,x)$. \\
Furthermore, if $v\in \NN(C,x)$ then Remark \ref{rem:hypo} yields that (\ref{eq:hypobis}) is verified.
\end{proof}

\begin{cor}
For $C$ a closed subset of $M$ and $x\in C$, then $$N^P_C(x) = \NN(C,x).$$
\label{cor:eqcones}
\end{cor}

\begin{proof}
Indeed, \cite[Lemma 2.1]{HP} states that $v\in N^P_C(x)$ if and only if there is $\alpha>0$ such that 
$$ \langle v , \exp^{-1}_x(y) \rangle \leq \frac{\alpha}{2} d(x,y)^2,  $$ 
for every $y$ in a neighborhood of $x$. This property is equivalent to (\ref{eq:hypobis}).  We conclude to the equality of these cones thanks to Proposition \ref{prop:hypomo}. 
\end{proof}

Now let us specify the link between the different definitions of prox-regularity.

\begin{prop} \label{pro}  
Let $C$ be a closed subset of $M$. If $C$ is locally prox-regular (according to Definition \ref{def:prox} ) then $C$ is prox-regular at any point $\bar{x} \in C$ (in the sense of Definition \ref{def:proxHP} ). 
\end{prop}

\begin{proof} 
Let $\bar{x} \in C $ and $\BB_0 = B(\bar{x}, 1)  $, then $\bar{x} $ belongs to the convex ball $ \BB= B(\bar{x}, \varepsilon) $, with $\varepsilon= \min(1, \rho(\BB_0)/2) $. Moreover let $x, y  \in C\cap \BB$, it comes $$ d(x,y)\leq 2\varepsilon \leq   \rho(\BB_0) \leq   \rho(\BB), $$ 
since $\BB \subset \BB_0 $.
 Thanks to Lemma \ref{lem:hypo}, for every $v \in N(C,x)$, we have $$ \langle v , \exp^{-1}_x(y) \rangle \leq  E_\BB |v| d(x,y)^2 . $$
Let $w \in N^L_C(x)$  satisfying $|w| \leq \varepsilon$ and $w= \lim w_i$, where $w_i \in N^P_C(x_i)$ with $x_i \in C$ and  $x_i \to x$. For $i$ large enough, $x_i \in \BB$ and $ |w_i| \leq 2 \varepsilon$, hence  $$ \langle w_i , \exp^{-1}_{x_i}(y) \rangle_{x_i} \leq   2 E_\BB \varepsilon  d(x_i,y)^2 . $$
Furthermore, 
\begin{eqnarray*}
  & & \langle w_i , \exp^{-1}_{x_i}(y) \rangle_{x_i} - \langle w, \exp^{-1}_{x}(y) \rangle_{x} \\
 & & \hspace{2cm}  =  \langle L_{x_i\to x} (w_i), L_{x_i\to x}(\exp^{-1}_{x_i}(y)) \rangle_{x}  - \langle w, \exp^{-1}_{x}(y) \rangle_{x}\\
   & & \hspace{2cm}  =  \langle L_{x_i\to x} (w_i)- w, \exp^{-1}_{x}(y) \rangle_{x}  \\
& & \hspace{4cm}  +\langle L_{x_i\to x} (w_i),   L_{x_i\to x}(\exp^{-1}_{x_i}(y))- \exp^{-1}_{x}(y) \rangle_{x}\\
 & & \hspace{2cm}  \to 0. 
\end{eqnarray*}
Indeed the second term tends to zero since $ x \mapsto \exp^{-1}_{x}(y)$ is smooth on $\BB$  thanks to Remark \ref{rem:smoothexp}. 
Thus by passing to the limit, the previous inequality becomes: 
$$\langle w , \exp^{-1}_{x}(y) \rangle_{x} \leq   2 E_\BB \varepsilon  d(x,y)^2  .$$
So we conclude that $C$ is prox-regular at $\bar{x}$. 
\end{proof}

\begin{rem}
In fact, a closed subset $C\subset M$ is prox-regular at $\bar{x} \in C$ according to \cite{HP} if and only if it satisfies a hypomonotonicity property in a convex neighborhood of $\bar{x}$ for every vector of  $\NN(C,\cdot)=N^P_C$. \\
Moreover, if $C\subset M$ is prox-regular at $x \in C$, it is proved in \cite[Lemma 3.7]{HP} that the limiting cone, the Clarke cone (that is the closed convex hull of the limiting cone) and the proximal normal cone are equal : $N^L_C(x) =N^C_C(x) = N^P_{C}(x) = \NN(C,x)$ (where the last equality comes from Corollary \ref{cor:eqcones}). 
\end{rem}

Then we aim to end this section by proving the following property (which is also well-known in the Euclidean space): close enough to a locally prox-regular subset, every point has a unique projection onto it.

\begin{thm} \label{thm:projection}
 Let $C \subset M$ be a locally prox-regular set. Then for every bounded set ${\mathcal B} \subset M$ satisfying $\BB \cap C \neq \emptyset $, there exists $\ell(\BB) >0$ such that for any $x\in {\mathcal B}$ with $d_C(x) < \ell(\BB)$, the projection set is a singleton:
$$ \PPP_C(x) =\{  x^\star\}, $$
with $ x^\star \in C.$
\end{thm}

\begin{proof}
Let ${\mathcal B}_1:= \{ m \in M, d(m, {\mathcal B})\leq 1\}$ and $\ell(\BB) :=  \min(\rho({\mathcal B}_1)/2, \tau(\BB_1), 1)$ where $\tau(\BB_1) := \frac{A(\BB_1)}{8 E_{\BB_1}} $, (we recall that these constants are defined in Corollary \ref{cor:new} and in Lemma \ref{lem:hypo}). 
Consider $x\in{\mathcal B}$ with $d_C(x) < \ell(\BB)$, then for every $\epsilon >0$  satisfying $$d_C(x)+\epsilon < \ell(\BB) \leq    \rho(\BB_1)/2$$ we can find $y\in C$ such that
$$ d(x,y)^2 \leq d_C(x)^2+\epsilon^2.$$
Let us consider the map, defined on $M$ by
$ F:= d(x,\cdot)^2.$
Then  by Ekeland's variational principle, we can find $z:= z(\epsilon)\in C$ such that
\begin{itemize}
 \item $d(y,z) \leq 1$,
 \item $d(x,z)^2 \leq d(x,y)^2 \leq  d_C(x)^2+\epsilon^2$,
 \item and $z$ is the minimizer over $C$ of the function $\psi:= F + \epsilon d(\cdot,z)$.
\end{itemize}
Note that $d(x,z) \leq d_C(x)+\epsilon \leq \ell(\BB) \leq 1$ so $d(z, \BB) \leq 1 $ and $z \in \BB_1$. 
We denote $\Pi(x,\epsilon)$ the set of points $z\in C$ satisfying the three last conditions. Clearly, for $\epsilon_1> \epsilon_2$, we have $\Pi(x,\epsilon_2) \subset \Pi(x,\epsilon_1)$. 
Now consider $\epsilon>0$ and $z_1,z_2\in\Pi(x,\epsilon)$. Then $z_1,z_2\in C \cap \BB_1$ and by first-order optimality conditions, we are going to prove that for $i=1,2$, there exists $v_i \in T_{z_i} M$ with $ |v_i|_{T_{z_i} M} \leq c_{{\mathcal B}_1} \epsilon$ (for some constant $c_{{\mathcal B}_1} >0$) such that 
\begin{equation} w_i := \Gamma_{z_i,x} +v_i \in \NN(C, z_i). \label{eq:j} \end{equation}

\mb
{\bf Step 1 :} Proof of (\ref{eq:j}). \\
Let $\phi := U \subset M \rightarrow H$ be a local chart around $z_i$.
By definition, we locally have 
\begin{align*}
 z_i & := \argmin{p\in C} d(x,p)^2 + \epsilon d(p,z_i) \\
 &  = \phi^{-1} \left[\argmin{P\in \phi(C \cap U)} d(\phi^{-1}(X),\phi^{-1}(P))^2 + \epsilon d(\phi^{-1}(P),\phi^{-1}(Z_i))\right] \\
 & = \phi^{-1} \left[\argmin{P\in \phi(U)} d(\phi^{-1}(X),\phi^{-1}(P))^2 + \epsilon d(\phi^{-1}(P),\phi^{-1}(Z_i)) +  I_{\phi(C)}(P)\right]
\end{align*}
where $X=\phi(x) $, $P=\phi(p)$,  $Z_i =\phi(z_i)$ and $I_{\phi(C)} $ represents the indicatrix function of $\phi(C) $. \\
Since $\phi(C) $ is closed, the characteristic function  $I_{\phi(C)}$ is lower semicontinuous. 
So we have
$$ 0 \in \nabla d^2(\phi^{-1}(X),\phi^{-1}(\cdot)) (Z_i) + \epsilon \partial^C (d(\phi^{-1}(\cdot),\phi^{-1}(Z_i))) (Z_i) + \lambda, $$ for some $\lambda \in \NN_{\phi(C)}^C(Z_i)$, where $\partial^C $ and $\NN^C$ represent the Clarke subdifferential and the Clarke normal cone. We refer the reader to  \cite[Section 3.1.1]{BS} for this result. 

As a consequence, there exists $U_i\in H$ with $|U_i|_H\leq 1$ satisfying
$$ 0 = - 2 [D \phi^{-1}(Z_i)]^* (\Gamma_{z_i,x} ) + \epsilon U_i + \lambda,$$
according to Remark \ref{rem:nablay} .   
That implies
$$ [D \phi^{-1}(Z_i)]^* (\Gamma_{z_i,x}) - \frac{\epsilon}{2} U_i \in \NN_{\phi(C)}^C( Z_i).$$
By composing with $([D \phi^{-1}(Z_i)]^*)^{-1}$ we have 
$$ \Gamma_{z_i,x} - \frac{\epsilon}{2} ([D \phi^{-1}(Z_i)]^* )^{-1}(U_i) \in ([D \phi^{-1}(Z_i)]^*) ^{-1}[\NN_{\phi(C)}^C( Z_i)].$$
Moreover, 
\begin{align*} 
w \in ([D \phi^{-1}(Z_i)]^*)^ {-1}[\NN_{\phi(C)}^C(Z_i)] & \Longleftrightarrow  [D \phi^{-1}(Z_i)]^{*}(w) \in \NN_{\phi(C)}^C(Z_i) \\
 & \Longleftrightarrow \forall \tau \in T^C_{\phi(C)}(Z_i), \langle [D \phi^{-1}(Z_i)]^{*}(w) , \tau \rangle \leq 0
\end{align*}
since the Clarke tangent cone $ T^C_{\phi(C)}(Z_i)$ to $\phi(C)$ at $Z_i$ is the polar cone of $\NN^C_{\phi(C)}(Z_i)$. 
As a consequence, 
\begin{align*} 
w \in ([D \phi^{-1}(Z_i)]^*)^ {-1}[\NN_{\phi(C)}^C(Z_i)]  & \Longleftrightarrow \forall \tau \in T^C_{\phi(C)}(Z_i), \langle w , D \phi^{-1}(Z_i)[\tau] \rangle \leq 0 \\
 & \Longleftrightarrow w \in  \left[D \phi^{-1}(Z_i)(T^C_{\phi(C)}(Z_i))\right]^\circ, 
 \end{align*}
by definition of a polar cone. 
Thus, using the change of charts, we obtain that $$ \left[D \phi^{-1}(Z_i)(T^C_{\phi(C)}(Z_i))\right]^\circ=\left[T^C_{C}(z_i)\right]^\circ = \NN^C_C(z_i),.$$
where we have used the definition of the Clarke tangent cone in \cite{HP} and the fact that the  Clarke normal cone is exactly its polar cone. 
Since the set is locally prox-regular, we know that $w \in \NN(C,z_i) $ thanks to Proposition \ref{pro}. 
We conclude the proof by choosing $v_i=-\frac{\epsilon}{2} ([D \phi^{-1}(Z_i)]^*)^ {-1}(U_i)$. In that case, $w_i := \Gamma_{z_i,x} +v_i \in \NN(C, z_i)$ and we have $|v_i| \leq  c_{\BB_1} \epsilon$ (the constant $c_{\BB_1}$ depending on $\phi$ and $\BB_1$).

\mb
{\bf Step 2 :} End of the proof. \\
Since $|w_i| \leq d(z_i, x) +|v_i| \leq d_C(x) + \epsilon + c_{\BB_1} \epsilon \leq \ell(\BB)+c_{{\mathcal B}_1}\epsilon$, 
and $ d(z_1,z_2) \leq  d(z_1, x) + d(z_2, x) \leq \rho(\BB_1)$ from Lemma \ref{lem:hypo}, it comes
$$ \langle w_1,  \Gamma_{z_1,z_2} \rangle_{T_{z_1}M} + \langle w_2,  \Gamma_{z_2,z_1} \rangle_{T_{z_2}M} \leq  2 E_{\BB_1}  ( \ell(\BB) +c_{\BB_1} \epsilon ) d(z_1,z_2)^2, $$ 
where $E_{\BB_1}   >0$ . 
Using the parallel transport, 
$$\langle w_1,  \Gamma_{z_1,z_2} \rangle_{T_{z_1}M} = \langle L_{z_1\to z_2} (w_1), -\Gamma_{z_2,z_1} \rangle_{T_{z_2}M},$$
hence we conclude to
$$ \langle w_2-L_{z_1\to z_2} (w_1),  \Gamma_{z_2,z_1} \rangle_{T_{z_2}M} \leq  2 E_{\BB_1}  ( \ell(\BB)+ c_{\BB_1} \epsilon ) d(z_1,z_2)^2.$$
By definition of the $w_i$, it comes
$$ \langle \Gamma_{z_2,x} +v_2-L_{z_1\to z_2} (\Gamma_{z_1,x} +v_1),  \Gamma_{z_2,z_1} \rangle_{T_{z_2}M} \leq 2 E_{\BB_1} ( \ell(\BB)+ c_{\BB_1}\epsilon ) d(z_1,z_2)^2$$
which yields
$$ \langle \Gamma_{z_2,x} -L_{z_1\to z_2}(\Gamma_{z_1,x}),  \Gamma_{z_2,z_1} \rangle_{T_{z_2}M}  \leq   2 E_{\BB_1}  ( \ell(\BB)+ c_{{\mathcal B}_1}\epsilon )  d(z_1,z_2)^2 + 2 c_{\BB_1} \epsilon d(z_1,z_2).$$

As $z_1, z_2 \in B(x, \ell(\BB))$ with $ \ell(\BB)< \rho(\BB_1)$,  we deduce from Corollary \ref{cor:new} that $$ \langle \Gamma_{z_2,x} -L_{z_1\to z_2}(\Gamma_{z_1,x}),  \Gamma_{z_2,z_1} \rangle_{T_{z_2}M}  \geq A(\BB_1) d(z_1,z_2)^2.$$ 
Hence we observe that

$$\left(A(\BB_1)- 2 E_{\BB_1}  \ell(\BB)-2 E_{\BB_1}   c_{\BB_1} \epsilon\right) d(z_1,z_2) \leq 2 c_{\BB_1}\epsilon$$
If $\epsilon < \frac{\tau(\BB_1)}{ c_{\BB_1} }$,  then $A(\BB_1)- 2 E_{\BB_1}  \ell(\BB)-2 E_{\BB_1}   c_{\BB_1} \epsilon \geq A(\BB_1)/2  > 0 $ since $ \ell(\BB) \leq \tau(\BB_1)  = \frac{A(\BB_1)}{8 E_{\BB_1} }$. 
As a consequence 
\begin{equation}
d(z_1,z_2) \leq \frac{2 c_{\BB_1}}{\left(A(\BB_1)- 2 E_{\BB_1}  \ell(\BB)-2 E_{\BB_1} c_{\BB_1}\epsilon\right)} \epsilon := \kappa \epsilon. 
\label{eq:dz1z2}
\end{equation}

Let $\epsilon_n \downarrow 0 $ and $z_n \in \Pi(x,\epsilon_n) $. It then follows from that for $m\geq n$, $$ d(z_n, z_m) \leq  \kappa \epsilon_n, $$ hence $(z_n)$ is a Cauchy sequence that converges to a point $z^\star \in C$ ($M$ is metrically complete). We obtain $z^\star \in \PPP_C(x)$ since $d(x,z^\star) = d_C(x)$, consequently $ \PPP_C(x) \neq \emptyset $. Furthermore,  taking $\epsilon =0$ in \eqref{eq:dz1z2}, we can specify that
 $ \PPP_C(x) =\{ z^\star\}$. 
\end{proof}

\begin{rem}
Note that the constant $\ell(\BB)$ depends on the geometrical properties of $M$ but also on $C$ through the term $E_\BB$ (which is determined by the quantities $\eta_\BB$ see Remark \ref{rem:constantE} ). 
\label{rem:constantell}
\end{rem}


\section{Sweeping process in a Hilbert manifold} \label{sec:sweep}

In this section, we consider a Riemannian Hilbert manifold $M$ satisfying Assumption \ref{ass:ic}. 

\begin{ass} \label{ass:C}  Let $C:[0,T ] \rightrightarrows M $ be a multivalued map satisfying that 
for all $t \in [0,T ], $  the set $C(t)$ is nonempty and locally prox-regular. More precisely, for all bounded set $\BB \subset M$, there exists $0 < \eta_\BB  \leq \rho(\BB)$ such that for every $t \in [0,T ]$, $s \in [0, \eta_\BB)$, $x \in C(t)$ and $v\in N(C(t), x) \setminus \{0\}$,  $$ x \in \PPP_{C(t)}\left(\gamma_{x, \frac{v}{|v|}}(s)\right).$$ Furthermore, $C$ is supposed to be a Lipschitz map: for all $t,s \in [0,T ]$, $d_H(C(t),C(s)) \leq K_L |t-s|$. 
\end{ass}

\begin{ass} \label{ass:f}
 We assume that $f :(t,x) \in [0,T] \times M \mapsto f(t,x) \in T_xM $ is bounded and satisfies a Lipschitz regularity
 in the following sense: 
 $$ g : \left\{ \begin{array}{ccc}
                 [0,T] \times M & \rightarrow & TM \\
                 (t,x) & \mapsto & (x,f(t,x))
                \end{array} \right.$$
is Lipschitz with constant $L_f$. More precisely, for all bounded set $\BB$ of $M$,  for all $(t_1, t_2)
\in [0,T]^2$ and $ (x_1, x_2)
\in \BB^2$ verifying $d(x_1, x_2) < \rho(\BB)$ , 
\begin{align*}
 d_{TM}(g(t_1,x_1), g(t_2,x_2)) &:=d(x_1,x_2) + | L_{x_1 \to x_2} (f(t_1,x_1)) - f(t_2,x_2)| \\
&\leq L_f (|t_1-t_2| + d(x_1,x_2) ).
\end{align*}
\end{ass}

\begin{thm}  Under Assumptions \ref{ass:ic}, \ref{ass:C} and \ref{ass:f}, for every $x_0\in M$, for all $T>0$,  there exists a unique solution $x \in W^{1,\infty}([0,T], M)$ such that
\be{eq:sys} 
\left\{ \begin{array}{ll}
 \dot{x}(t)+\NN(C(t),x(t)) \ni f(t,x(t))  \textmd{ for a.e. } t \in  [0,T] \vsp  \\
 x(0)=x_0.
\end{array} \right. \ee
\end{thm}

\mb The proof is a mixture of the classical one in a Hilbert space (see the papers cited in the introduction) and of the previous results extending arguments in a Riemannian context. 

\begin{proof}
 
\mb {\bf First step: } Construction of discrete solutions. \\
We set $\K_0:=M\cap
B(x_0, 2 T\|f\|_\infty + K_L T )$. We can choose $\bar T \leq T$ such that $ 2 \bar T \|f\|_\infty + K_L \bar T< \min(\eta_{\BB_0}/2, \ell(\BB_0) )$, where $\ell(\BB_0)$ is defined in Theorem \ref{thm:projection}  by replacing $C$ with $C(0)$. 
Note that the quantities $\eta_{\BB_0}$,  $\ell(\BB_0) $ are the same for all the sets $C(t), t \in [0,T]$ thanks to Remark \ref{rem:constantell} and Assumption \ref{ass:C}. \\
Then we define $\K:=M\cap
B(x_0, 2 \bar T\|f\|_\infty + K_L \bar T ) \subset \K_0$, hence $\rho(\K)/2 \geq \rho(\K_0)/2 $.  As a consequence, it yields $$2 \bar T \|f\|_\infty + K_L \bar T< \eta_{\BB_0}/2  \leq   \rho(\BB_0)/2 \leq \rho(\K)/2.$$   
We fix a time-step $h=\bar T/n$ which obviously satisfies
\be{eq:1} h \|f\|_\infty  \leq \rho(\K)/2  \ee
So we consider a subdivision of the time-interval $ J=[0,\bar T]$ defined by $t^n_i = ih$ for $i\in\{0,..,n\}$. We build $(x^n_i)_{0\leq i\leq n}$ as follows:
\be{scheme} \left\{ \begin{array}{l}
                x^n_0=x_0 \vsp \\
                x^n_{i+1}\in \PPP_{C(t^n_{i+1})}\left[\gamma_{x^n_i,f(t^n_i,x^n_i)}(h)\right].
           \end{array} \right. \ee
First let us check that this scheme is well-defined. 
  \begin{lem} \label{lem:schemewp}
  For all $i \in \{0, \dots n\}$,   $ x_i^n$ and  for all $i \in \{0, \dots n-1\}$ $\gamma_{x^n_i,f(t^n_i,x^n_i)}(h) $ are well-defined and belong to $\BB$.  Furthermore for $i \in \{0, \dots n-1\} $ , $d(x^n_i, x_{i+1}^n) \leq \rho(\BB)/2$ and  $ d( x_{i+1}^n, \gamma_{x^n_i,f(t^n_i,x^n_i)}(h) ) \leq \rho(\BB)/2$.
  \end{lem}
  \begin{proof}
  First $x^n_0= x_0 \in \BB$ and since $h \|f\|_\infty \leq \rho(\BB)$, $\gamma_{x^n_0,f(t^n_0,x^n_0)}(h) $ exists.
  Moreover $ d(x_0, \gamma_{x^n_0,f(t^n_0,x^n_0)}(h) )  \leq h \|f\|_\infty $ hence $\gamma_{x^n_0,f(t^n_0,x^n_0)}(h)  \in \BB$. 
  Assume that for some $i<n$, $ x_i^n$ and $\gamma_{x^n_i,f(t^n_i,x^n_i)}(h) $ are well-defined and belong to $\BB$.
Let us show that  $x_{i+1}^n $ and  (if $i<n-1$)  $\gamma_{x^n_{i+1},f(t^n_{i+1},x^n_{i+1})}(h) $ satisfy the same properties. 
  Since $ d(x_i^n,  \gamma_{x^n_i,f(t^n_i,x^n_i)}(h)  )  \leq h \|f\|_\infty $ and $d_{ C(t_{i+1}^{n})}(x_i^n) \leq d_H(C(t_i^{n}), C(t_{i+1}^{n}) \leq K_Lh $ then 
\begin{align*}
d_{C(t_{i+1}^{n})}(\gamma_{x^n_i,f(t^n_i,x^n_i)}(h)) &\leq d(x_i^n, \gamma_{x^n_i,f(t^n_i,x^n_i)}(h)) + d_{C(t_{i+1}^{n})}(x_i^n) \\
&\leq h(\|f\|_\infty +K_L) \leq \ell(\BB_0) . 
\end{align*}
In addition $ \BB_0 \cap C(t_{i+1}^{n}) \neq \emptyset $. Otherwise, for all $z \in C(t^n_{i+1})$,  $d(x_0, z) \geq 2 T \|f\|_\infty + K_L T$ and so $d_{C(t^n_{i+1})} (x_0) \geq 2 T \|f\|_\infty + K_L T $.  Yet $d_{C(t^n_{i+1})} (x_0) \leq d_H(C(0),C(t^n_{i+1}) ) \leq K_L T $. 
Consequently $\PPP_{C(t^n_{i+1})}\left[\gamma_{x^n_i,f(t^n_i,x^n_i)}(h) \right ] \neq \emptyset $ according Theorem \ref{thm:projection}, so $ x_{i+1}^n$ exists. 
  Moreover for $ 0 \leq j \leq i$, 
\begin{align*}
d(x_j^n, x_{j+1}^n)& \leq d(x_j^n, \gamma_{x^n_j,f(t^n_j,x^n_j)}(h)) + d(x_{j+1}^n, \gamma_{x^n_j,f(t^n_j,x^n_j)}(h)) \\
&\leq h  \|f\|_\infty + d_{C(t_{j+1}^{n})}(\gamma_{x^n_j,f(t^n_j,x^n_j)}(h))  \\
&\leq h(2  \|f\|_\infty +K_L) \leq \rho(\BB)/2.  
\end{align*}
 Thus $d(x_0, x_{i+1}^n) \leq (i+1)h (2  \|f\|_\infty +K_L) \leq  \bar{T}(2  \|f\|_\infty +K_L), $ hence $ x_{i+1}^n$ belongs to $\BB$ and $$d(x_{i+1}^n,  \gamma_{x^n_{i},f(t^n_{i},x^n_{i})}(h)  )   = d_{C(t_{i+1}^{n})}(\gamma_{x^n_i,f(t^n_i,x^n_i)}(h)) \leq h(\|f\|_\infty +K_L)  \leq  \rho(\BB)/2.$$  
Since $ h \|f\|_\infty  \leq \rho(\BB)/2$, $ \gamma_{x^n_{i+1},f(t^n_{i+1},x^n_{i+1})}(h)  $  is well-defined.
Finally  if $i<n-1$,  
$$\begin{array}{lcl}
 d(x_0, \gamma_{x^n_{i+1},f(t^n_{i+1},x^n_{i+1})}(h)) & \leq & d(x_0, x_{i+1}^n) + d(x_{i+1}^n, \gamma_{x^n_{i+1},f(t^n_{i+1},x^n_{i+1})}(h)) \vsp \\
  &\leq & (i+2)h (2  \|f\|_\infty +K_L) , \vsp \\
&\leq & nh (2  \|f\|_\infty +K_L) \leq  \bar{T}(2  \|f\|_\infty +K_L),
  \end{array}$$
  
   so $ \gamma_{x^n_{i+1},f(t^n_{i+1},x^n_{i+1})} \in \BB$.  The proof is also ended. 
  \end{proof}
Now thanks to the points $(x^n_i)_{0\leq i\leq n}$ we define two
piecewise maps $x^n$ and $f^n$ on $J$ in the following way:
\be{eqf} \forall t \in J_i:=[ih, (i+1)h[ \virg f^n(t):=f (t^n_i ,x^n_i) \in T_{x^n_i}M, f^n(\bar T):=f (t^n_n ,x^n_n) \ee
and
\begin{align}
\forall t \in J_i:=[ih, (i+1)h[ \virg x^n(t) &:=
\gamma_{x^n_i,\Gamma_{x^n_i,x^n_{i+1}}}\left(\frac{t-ih}{h}\right),  x^n(\bar T ):=x^n_n.
\label{eqxn}
\end{align}
The function $x^n$ is continuous on $J=[0,\bar T]$.
Note that $\Gamma_{x^n_i,x^n_{i+1}} $ is well-defined since $d(x^n_{i+1},x^n_i ) <\rho(\BB)/2 $ (according to Lemma  \ref{lem:schemewp}).  

Moreover we define two other
piecewise maps $\tau^n$ and $\theta^n$ on $J$ in the following way:
$$\forall t \in J_i:=[ih, (i+1)h[ \virg \tau^n(t):=t^n_i,   \tau^n(\bar T):=\bar T \textmd{ and
} \theta^n(t):=t^n_{i+1}, \theta^n(\bar T):= \bar T . $$

\mb {\bf Second step: } Boundedness of the discretized velocities.\\
We claim that the discretized velocities are uniformly bounded. Indeed
for $n$, $i\in\{0,...,n-1\}$ and $t\in J_i$, the scheme (\ref{scheme})
with (\ref{eqxn}) gives
\begin{align}
 |\dot x^n(t)| & = \frac{1}{h} \left|\dot \gamma_{x_i^n,\Gamma_{x_i^n,x_{i+1}^n}}\left(\frac{t-ih}{h}\right)\right| \nonumber \\
  & \leq \frac{1}{h} |\Gamma_{x_i^n,x_{i+1}^n}| = \frac{d(x_i^n,x_{i+1}^n)}{h} \nonumber \\
  &  \leq 2 \|f\|_{\infty} + K_L. \label{eq:vit}
\end{align}
So, let 
\be{eq:V} V:= \sup_{i,n} h^{-1} d(x^n_{i+1},x^n_i) \leq 2 \|f\|_\infty +K_L <\infty. \ee
 Thus we observe
 that \be{eq:borne} \forall n \ \textrm{and} \ \forall t, \quad  x^n(t) \in \K. \ee 

\mb {\bf Third step: } Differential inclusion for the discrete solutions. \\
We  are now looking for a discrete differential inclusion satisfied by the function $x^n$. More precisely, we want to check that for almost every $i\in\{0,...,n-1\}$, we have
\be{inclu:amontrer}  \Gamma_{x^n_{i+1},x^n_i} + h f(t^n_{i+1},x^n_{i+1}) +\RR(h^2)\in \NN(C(t^n_{i+1}),x^n_{i+1}),\ee
with $\RR(h^2) \in T_{x^n_{i+1}} M $ and $ |\RR(h^2)| =\OO(h^2) $.
By definition of the scheme (\ref{scheme}), we know that (taking $\epsilon=1$ in the definition of proximal normal vectors and writing $ \gamma_{x^n_i,f(t^n_i,x^n_i)}(h) = \gamma_{x^n_i,v}(1) $ where $v= \Gamma_{x^n_{i+1},\gamma_{x^n_i,f(t^n_i,x^n_i)}(h)}  $ )
\be{aa} \Gamma_{x^n_{i+1},\gamma_{x^n_i,f(t^n_i,x^n_i)}(h)} \in \NN(C(t^n_{i+1}),x^n_{i+1}).\ee
Note that $\Gamma_{x^n_{i+1},\gamma_{x^n_i,f(t^n_i,x^n_i)}(h)}$ is well-defined since $d(x^n_{i+1},\gamma_{x^n_i,f(t^n_i,x^n_i)}(h)) \leq \rho(\BB)/2$ (according to Lemma \ref{lem:schemewp}).
Thanks to the smoothness of the exponential map and of the geodesics, let us check the following equality:  
\begin{lem}
\be{eq:r1} \Gamma_{x^n_{i+1},\gamma_{x^n_i,f(t^n_i,x^n_i)}(h)} = \Gamma_{x^n_{i+1},x^n_i} + h D  \exp_{x^n_{i+1}}^{-1}(x^n_i) [f(t^n_i,x^n_i)] + \RR_1 (h^2),\ee
\end{lem}
  \begin{proof}
By setting $y^n_i:=\gamma_{x^n_i,f(t^n_i,x^n_i)}(h) $, we recall that 
$$\Gamma_{x^n_{i+1},\gamma_{x^n_i,f(t^n_i,x^n_i)}(h)} = \exp_{x^n_{i+1}}^{-1}(y^n_i).$$
Now we define for $ s \in [0,h]$,  $\gamma(s):= \exp_{x^n_{i}}(s f(t^n_i,x^n_i) )$, so   by using a first order expansion, it comes 
$$ \exp_{x^n_{i+1}}^{-1}(y^n_i) = \exp_{x^n_{i+1}}^{-1}(x^n_i) + \int_{0}^h D\exp_{x^n_{i+1}}^{-1}(\gamma(s))[\dot{\gamma}(s)] ds. $$ 
Furthermore
\begin{align*}
&\left |\int_{0}^h D\exp_{x^n_{i+1}}^{-1}(\gamma(s))[\dot{\gamma}(s)] ds - h  D  \exp_{x^n_{i+1}}^{-1}(x^n_i) [f(t^n_i,x^n_i)]  \right |& \\
\quad &= \left |\int_{0}^h D\exp_{x^n_{i+1}}^{-1}(\gamma(s))[\dot{\gamma}(s)] ds  - \int_{0}^h  D  \exp_{x^n_{i+1}}^{-1}(\gamma(0)) [\dot{\gamma}(0)] ds\right |  &\\
\quad & \leq\int_{0}^h\left | D\exp_{x^n_{i+1}}^{-1}(\gamma(s))[\dot{\gamma}(s)] ds  -   D  \exp_{x^n_{i+1}}^{-1}(\gamma(0)) [\dot{\gamma}(0)] \right | ds &\\
\quad & \leq \int_{0}^h C_e(\BB)  d_{TM}( (\gamma(s),\dot{\gamma}(s) ) , (\gamma(0), \dot{\gamma}(0)) )  ds \textmd{  according to Assumption \ref{ass:ic} } &\\
\quad & \leq C_e(\BB) \int_{0}^h  (d(\gamma(0),\gamma(s) ) +|  L_{\gamma(s) \to\gamma(0) }( \dot{\gamma}(s))  -  \dot{\gamma}(0)| )ds  &\\
\quad & \leq C_e(\BB) \int_{0}^h ( s |f(t^n_i,x^n_i) | +  L(\BB)   d(\gamma(0),\gamma(s) ) ^2) ds \textmd{  using Lemma \ref{lem:geodesic}   } &\\
\quad & \leq C_e(\BB) (\|f\|_{\infty} h^2  +  L(\BB) \|f\|_{\infty}^2 h^3 ) . 
\end{align*}
Thus fixing $\RR_1 (h^2) := \dsp \int_{0}^h D\exp_{x^n_{i+1}}^{-1}(\gamma(s))[\dot{\gamma}(s)] ds - h  D  \exp_{x^n_{i+1}}^{-1}(x^n_i) [f(t^n_i,x^n_i)]  $, we can write 
$$\exp_{x^n_{i+1}}^{-1}(y^n_i) = \exp_{x^n_{i+1}}^{-1}(x^n_i)  + h D  \exp_{x^n_{i+1}}^{-1}(x^n_i) [f(t^n_i,x^n_i)] + \RR_1 (h^2) , $$
which is the desired result.

\end{proof}

With Assumption \ref{ass:ic}, the map $D\exp^{-1}_{x^n_{i+1}} $ is Lipschitz and so it comes (using the notations introduced in this assumption)
\begin{align*}
 & \left| D  \exp_{x^n_{i+1}}^{-1}(x^n_i) [f(t^n_i,x^n_i)] - D
   \exp_{x^n_{i+1}}^{-1}(x^n_{i+1}) [f(t^n_{i+1},x^n_{i+1})] \right|  \\
  & \quad \leq C_e(\K) d_{TM}\left(g(t^n_i,x_n^i),g(t^n_{i+1},x^n_{i+1}) \right) \\
   &  \quad  \leq C_e(\K) L_f (h+d(x^n_i,x^n_{i+1})) \\
   &\quad \leq C_e(\K) L_f (1+V)h  := C_1 h ,
\end{align*}
for some numerical constant $C_1 >0$, where we used Assumption (\ref{ass:f})
(on $g$ introduced there) and (\ref{eq:V}). Thus (\ref{eq:r1}) becomes
\begin{align*}
 \Gamma_{x^n_{i+1},\gamma_{x^n_i,f(t^n_i,x^n_i)}(h)}  & = \Gamma_{x^n_{i+1},x^n_i} + h D  \exp_{x^n_{i+1}}^{-1}(x^n_{i+1}) [f(t^n_{i+1},x^n_{i+1})] + \RR (h^2) \\
  & = \Gamma_{x^n_{i+1},x^n_i} + h f(t^n_{i+1},x^n_{i+1}) + \RR (h^2).
\end{align*}
Since for every $x\in M$, $D  \exp_{x}^{-1} (x) = [D
\exp_x(0)]^{-1} = Id_{T_xM}$. From this and (\ref{aa}), we deduce
(\ref{inclu:amontrer}). 

\mb {\bf Fourth step: }  Cauchy sequence. \\
Let $n, m$ two integers, by (\ref{eq:borne}) for all $t \in J$,
$x^n(t)$ and $x^m(t)$ belong to $\K $ and so $ d(x^n(t),x^m(t)) < d(x^n(t),x_0) + d(x_0,x^m(t))<\rho(\K)/2 + \rho(\K)/2  < \rho(\K)$.
We consider $t \notin \{ \tau^n(t), \theta^n(t), \tau^m(t), \theta^m(t)  \}$.
From (\ref{inclu:amontrer}) and Lemma \ref{lem:hypo}, we get for
$y=x^m(t) \in \K$,
\begin{align*} 
& \langle \Gamma_{x^n(\theta^n(t)),x^n(\tau^n(t))} + h f(\theta^n(t),x^n(\theta^n(t)),\Gamma_{x^n(\theta^n(t)),x^m(t)} \rangle_{x^n(\theta^n(t))} \\
& \hspace{2cm} \leq {E_\K} |\Gamma_{x^n(\theta^n(t)),x^n(\tau^n(t))} + h f(\theta^n(t),x^n(\theta^n(t))) | d(x^n(\theta^n(t)),x^m(t))^2  \\ 
& \hspace{2.2cm} + \OO(h^2)d(x^n(\theta^n(t),x^m(t))) + \OO(h^2)d(x^n(\theta^n(t),x^m(t))^2. \end{align*}
Since 
\be{ineq} d(x^n(\theta^n(t)),x^m(t)) \leq  1 + d(x^n(\theta^n(t)),x^m(t))^2 ,\ee 
it comes 
\begin{align*}
 \langle \Gamma_{x^n(\theta^n(t)),x^n(\tau^n(t))} + h f(\theta^n(t),x^n(\theta^n(t))),\Gamma_{x^n(\theta^n(t)),x^m(t)} \rangle_{x^n(\theta^n(t))} \nonumber \\
   \leq {E_\K}h (V+\|f\|_{\infty}+\OO(h)) d(x^n(\theta^n(t)),x^m(t))^2 + \OO(h^2). 
\end{align*}
As \begin{align*} 
 d(x^n(\theta^n(t)),x^m(t)) &\leq d(x^n(\theta^n(t)),x^n(t)) +
 d(x^n(t),x^m(t)) \\   & \leq V h + d(x^n(t),x^m(t)) 
\end{align*} 
 we have  $$ d(x^n(\theta^n(t)),x^m(t))^2 \leq 2 d(x^n(t),x^m(t))^2 + \OO(h^2)
$$ and as a consequence
\begin{align}
 \langle \Gamma_{x^n(\theta^n(t)),x^n(\tau^n(t))} + h f(\theta^n(t),x^n(\theta^n(t))),\Gamma_{x^n(\theta^n(t)),x^m(t)} \rangle_{x^n(\theta^n(t))} \nonumber \\
  \leq {2E_\K}h (V+\|f\|_{\infty}+\OO(h)) d(x^n(t),x^m(t))^2 + \OO(h^2). \label{eq:i}
\end{align}
Let us split the left term in $I_n(t)+II_n(t)$ with $$I_n(t):= \langle \Gamma_{x^n(\theta^n(t)),x^n(\tau^n(t))},\Gamma_{x^n(\theta^n(t)),x^m(t)} \rangle_{x^n(\theta^n(t))} $$ and $$II_n(t):= h\langle  f(\theta^n(t),x^n(\theta^n(t))),\Gamma_{x^n(\theta^n(t)),x^m(t)} \rangle_{x^n(\theta^n(t))}.$$
By symmetry, with changing the role of $x^n$ and $x^m$, we obtain
\begin{align*}
  I_m(t)+II_m(t) & \leq {2E_\K}h (V+\|f\|_{\infty}+\OO(h)) d(x^n(t),x^m(t))^2 + \OO(h^2).
\end{align*}
Thus \be{eq:inm}  \frac{ I_n(t)+I_m(t)+II_n(t)+II_m(t)}{h}  \leq {4\A_\K}(V+\|f\|_{\infty}+\OO(h)) d(x^n(t),x^m(t))^2 + \OO(h).  \ee 
Let us study all the quantities $I_n,I_m,II_n$ and $II_m$. 

\mb
For $I_n$, we use the parallel transport from $x^n(\theta^n(t))$ to $x^n(t)$ to obtain
\begin{align}
 I_n(t) = \langle
\Gamma_{x^n(\theta^n(t)),x^n(\tau^n(t))},\Gamma_{x^n(\theta^n(t)),x^m(t)}
\rangle_{x^n(\theta^n(t))}  \nonumber \\
=- h\langle \dot x^n(t),L_{x^n(\theta^n(t)) \to
  x^n(t)}\left[\Gamma_{x^n(\theta^n(t)),x^m(t)}\right]\rangle_{x^n(t)}, \label{eq:in}
\end{align}
since 
\begin{align} L_{x^n(\theta^n(t)) \to x^n(t)}\left[\Gamma_{x^n(\theta^n(t)),x^n(\tau^n(t))}\right] &=\frac{d(x^n(\theta^n(t)),x^n(\tau^n(t)))}{d(x^n(t),x^n(\tau^n(t)))}\Gamma_{x^n(t),x^n(\tau^n(t))}
\nonumber \\
 &= - h\dot x^n(t). \label{eq:der}
\end{align}
Indeed the curve $x^n(t)$ is exactly the geodesic between $x^n(\tau^n(t))$ and $x^n(\theta^n(t))$ with velocity $d(x^n(\theta^n(t)),x^n(\tau^n(t))) /h$.  
Similarly for $m$, we have
 \be{eq:im} I_m(t) =- h \langle \dot x^m(t),L_{x^m(\theta^m(t)) \to x^m(t)}\left[\Gamma_{x^m(\theta^m(t)),x^n(t)}\right]\rangle_{x^m(t)}.\ee

\mb
Concerning $II_n$ and $II_m$, we split them in two parts  $II_n  = A_n+B_n$ with
\begin{align*}
A_n & :=  h \langle L_{x^n(\theta^n(t)) \to x^n(t)} \left[ f(\theta^n(t),x^n(\theta^n(t))) \right] ,  \\
&  \hspace{2cm} L_{x^n(\theta^n(t)) \to x^n(t)}[\Gamma_{x^n(\theta^n(t)),x^m(t)}]-\Gamma_{x^n(t),x^m(t)} \rangle_{x^n(t)}
\end{align*}
and
$$ B_n := \  h \langle  L_{x^n(\theta^n(t)) \to x^n(t)} \left[ f(\theta^n(t),x^n(\theta^n(t))) \right], \Gamma_{x^n(t),x^m(t)} \rangle_{x^n(t)}.$$
The first term $A_n$ is estimated as follows:
\begin{align}
 |A_n| & \leq h \|f\|_{\infty} |\Gamma_{x^n(t),x^m(t)}-L_{x^n(\theta^n(t)) \to x^n(t)}[\Gamma_{x^n(\theta^n(t)),x^m(t)}]|_{T_{x^n(t)}M}  \nonumber \\
 & \leq h \|f\|_{\infty} |\exp_{x^n(t)}^{-1}(x^m(t)) - L_{x^n(\theta^n(t)) \to x^n(t)}[\exp_{x^n(\theta^n(t))}^{-1}(x^m(t))] |_{T_{x^n(t)}M}  \nonumber \\
 &   \leq \frac{1}{2}  h \|f\|_{\infty}|\nabla_x d^2(x^n(t), x^m(t)) - L_{x^n(\theta^n(t)) \to x^n(t)}[\nabla_x d^2(x^n(\theta^n(t)), x^m(t)) ] |_{T_{x^n(t)}M}  \nonumber \\
 & \quad \quad \textmd{  according to Lemma \ref{lemb} }  \nonumber \\
& \leq  \frac{1}{2} h \|f\|_{\infty}  \sup_{s \in [t, \theta^n(t)]} \|H_x d^2(x^n(s), x^m(t))\|_{x^n(s)}  d(x^n(\theta^n(t)),x^n(t)) \nonumber \\
 & \leq   \frac{1}{2} C_{\rho(\K)} h \|f\|_{\infty} d(x^n(\theta^n(t)),x^n(t)) \leq  \frac{1}{2}  C_{\rho(\K)}  h^2 \|f\|_{\infty} V ,
 \label{eq:calculus}
\end{align}
 where we used Proposition \ref{propH} and (\ref{eq:V}).
Similarly for $A_m$, we obtain
\begin{align*}
| A_m |\leq  \frac{1}{2}  C_{\rho(\K)}  h^2 \|f\|_{\infty} V.
\end{align*}
Finally with (\ref{eq:inm}),(\ref{eq:in}) and (\ref{eq:im}), we obtain
\begin{align}
& -\langle \dot x^n(t),L_{x^n(\theta^n(t)) \to x^n(t)}\left[\Gamma_{x^n(\theta^n(t)),x^m(t)}\right]\rangle_{x^n(t)} \\ &  \hspace{2cm}- \langle \dot x^m(t),L_{x^m(\theta^m(t)) \to x^m(t)}\left[\Gamma_{x^m(\theta^m(t)),x^n(t)}\right]\rangle_{x^m(t)} \nonumber \\
& \hspace{3cm} \leq \frac{|B_n+B_m|}{h} +  {4E_\BB}(V+\|f\|_{\infty}+\OO(h)) d(x^n(t),x^m(t))^2 + \OO(h) \label{eq2}
\end{align}
It remains us to estimate $|B_n+B_m|$. With a parallel transport, we get
\begin{align*}
 \frac{B_n+B_m}{h}  & = \langle  L_{x^n(\theta^n(t)) \to x^n(t)}
 \left[ f(\theta^n(t),x^n(\theta^n(t))) \right], \Gamma_{x^n(t),x^m(t)}
 \rangle_{x^n(t)}  \\
&  \hspace{1.5cm} + \langle  L_{x^m(\theta^m(t)) \to x^m(t)} \left[
   f(\theta^m(t),x^m(\theta^m(t))) \right], \Gamma_{x^m(t),x^n(t)}
 \rangle_{x^m(t)} \\
 &=  \langle  L_{x^n(t) \to x^m(t)} L_{x^n(\theta^n(t)) \to x^n(t)}
 \left[ f(\theta^n(t),x^n(\theta^n(t))) \right],- \Gamma_{x^m(t),x^n(t)}
 \rangle_{x^m(t)} \\
& \hspace{1.5cm} + \langle  L_{x^m(\theta^m(t)) \to x^m(t)} \left[
  f(\theta^m(t),x^m(\theta^m(t))) \right], \Gamma_{x^m(t),x^n(t)}
\rangle_{x^m(t)},
\end{align*} 
since $  L_{x^n(t) \to x^m(t)} (\Gamma_{x^n(t),x^m(t)})
=-\Gamma_{x^m(t),x^n(t)}. $
Consequently, 
\begin{align*}
 \frac{B_n+B_m}{h} & = \langle  L_{x^n(t) \to x^m(t)} L_{x^n(\theta^n(t)) \to x^n(t)}
  \left[ f(\theta^n(t),x^n(\theta^n(t))) \right]  \\
& \hspace{1.5cm} - L_{x^m(\theta^m(t)) \to x^m(t)} \left[
  f(\theta^m(t),x^m(\theta^m(t))) \right] , -\Gamma_{x^m(t),x^n(t)}
\rangle_{x^m(t)} \\
 &\leq | L_{x^n(t) \to x^m(t)} L_{x^n(\theta^n(t)) \to x^n(t)} \left[
   f(\theta^n(t),x^n(\theta^n(t))) \right] \\
& \hspace{1.5cm} - L_{x^m(\theta^m(t)) \to x^m(t)} \left[
  f(\theta^m(t),x^m(\theta^m(t))) \right]|_{T_{x^m(t)}M} \ 
d(x^m(t),x^n(t)).
\end{align*}
We estimate the last norm by making appear intermediate
points. Indeed, using the Lipschitz regularity on $f$ (Assumption \ref{ass:f}) and (\ref{eq:V}), we have
\begin{align*}
 \left| f(t, x^m(t)) - L_{x^m(\theta^m(t)) \to x^m(t)} \left[ f(\theta^m(t),x^m(\theta^m(t)) \right] \right| \\
 \leq L_f (h + d( x^m(t),x^m(\theta^m(t)) )) \\
\leq L_f (1+V) h  \lesssim h.
\end{align*}
Similarly, 
$$ \left| L_{x^n(t) \to x^m(t)} \left(L_{x^n(\theta^n(t)) \to x^n(t)} \left[ f(\theta^n(t),x^n(\theta^n(t))) \right] -  f(t,x^n(t)) \right)\, \right| \lesssim h$$
and 
$$ \left| L_{x^n(t) \to x^m(t)} \left[f(t,x^n(t))\right] - f(t, x^m(t))\right| \lesssim d(x^n(t),x^m(t)).$$
So we can conclude that
$$ \begin{array}{lll}
\dfrac{|B_n+B_m|}{h} &\lesssim & \left(h+ d(x^n(t),x^m(t)) \right) d(x^m(t),x^n(t))\vsp \\
&\lesssim &  h d(x^n(t),x^m(t)) + d(x^m(t),x^n(t))^2 \vsp \\
&\lesssim &  h (1 +d(x^m(t),x^n(t))^2) +  d(x^m(t),x^n(t))^2 \textmd{ with} (\ref{ineq}). \vsp \\
\end{array}
$$
Finally, (\ref{eq2}) becomes
\begin{align}
 & - \langle \dot x^n(t), L_{x^n(\theta^n(t)) \to x^n(t)} \left[ \Gamma_{x^n(\theta^n(t)),x^m(t)}\right] \rangle_{x^n(t)} \nonumber
\\ &  \hspace{2cm} - \langle \dot x^m(t),  L_{x^m(\theta^m(t)) \to x^m(t)}  \left[ \Gamma_{x^m(\theta^m(t)),x^n(t)}\right] \rangle_{x^m(t)} \nonumber \\ 
 & \hspace{5cm} \leq (C_2+ \OO(h)) d(x^n(t),x^m(t))^2 + \OO(h) \label{eq3}, 
\end{align}
for some numerical constant $C_2$ (not depending on $n$ and $m$).
For the first terms, since the discretized velocities are uniformly
bounded, it follows that
$$ \begin{array}{l}
|L_{x^n(\theta^n(t)) \to x^n(t)}\left[\Gamma_{x^n(\theta^n(t)),x^m(t)}\right] -
\Gamma_{x^n(t),x^m(t)} | \vsp \\ \hspace{0.5cm} = |L_{x^n(\theta^n(t)) \to x^n(t)}\left[\exp^{-1}_{x^n(\theta^n(t))}(x^m(t))\right] -
\exp^{-1}_{x^n(t)}(x^m(t))|  \vsp \\ \hspace{0.5cm} \leq \frac{1}{2} C_{\rho(\BB)} d(x^n(\theta^n(t)), x^n(t))
\lesssim h  \textmd{ (same arguments as in the proof of \eqref{eq:calculus})}.
\end{array}$$
Hence 
\begin{align*}
-\langle \dot x^n(t),L_{x^n(\theta^n(t)) \to x^n(t)}\left[\Gamma_{x^n(\theta^n(t)),x^m(t)}\right]\rangle_{x^n(t)}
\\  =  -\langle \dot x^n(t),\Gamma_{x^n(t),x^m(t)}\rangle_{x^n(t)} +\OO(h).
\end{align*}
Producing similar reasoning for the symmetrical quantity, it comes
\begin{align}
-\langle \dot x^n(t),\Gamma_{x^n(t),x^m(t)}\rangle_{x^n(t)} - \langle \dot x^m(t),\Gamma_{x^m(t),x^n(t)} \rangle_{x^m(t)} \nonumber \\
\leq (C_2+ \OO(h)) d(x^n(t),x^m(t))^2 + \OO(h) \label{eq4},
\end{align}
which gives (according to Lemma \ref{lemb})
\begin{align*}
 \langle \dot x^n(t),\nabla_x d^2(x^n(t),x^m(t))\rangle_{x^n(t)} + \langle \dot x^m(t),\nabla_x d^2(x^m(t),x^n(t)) \rangle_{x^m(t)} \nonumber \\
\leq 2(C_2+ \OO(h)) d(x^n(t),x^m(t))^2 + \OO(h).
\end{align*}
Since $ \nabla_x d^2(x^m(t),x^n(t)) = \nabla_y d^2(x^n(t),x^m(t)) $, we conclude that 
$$ \frac{d}{dt} d(x^n(t) ,x^m(t))^2  \leq 2(C_2+ \OO(h)) d(x^n(t),x^m(t))^2 + \OO(h).$$
As usual, Gronwall's Lemma implies that
$$  \sup_{t\in [0,\bar{T}]} d(x^n(t),x^{m}(t))^2 = \OO(h).$$
and so $(x^n)_{n}$ is a Cauchy sequence of $C^0([0,T],M)$. Since $M$ is supposed to be metrically complete, the sequence $(x^n)_{n}$ also strongly converges to a function $x\in C^0([0,\bar{T}],M)$.
Furthermore \be{eq:lim}  \sup_{t\in [0,\bar{T}]} d(x^n(t),x(t)) = \OO(h^\frac{1}{2})\ee
\mb {\bf Fifth step: }  The limit function $x$ is solution of (\ref{eq:sys}). \\

Let $t_0 \in [0, \bar T]$,  there exists $\Phi:U_0
\rightarrow H$ a chart where $U_0$ is an bounded open set of $M$ containing $x(t_0)$. As $U_0$ is an open set, $$  \exists r < \rho(U_0), \, U:= B(x(t_0),r) \subset U_0. $$ Obviously,  $r < \rho(U_0) \leq \rho(U)$.  
Since $x$ is continuous on $[0, T]$, there is an open interval $\tilde{J} \subset [0, \bar T]$ containing $t_0$ such that 
\begin{align}
\forall t \in \tilde{J}, \, d(x(t),x(t_0)) < \frac{r}{4}  . 
\label{hypo:J}
\end{align}
Moreover as $x^n$ uniformly converges to $x$ on  $[0, \bar T]$, there exists $h_0$ such that for all $h<h_0$,  for all $t \in J$

\begin{align}
&d(x^n(t),x(t)) < \frac{r}{4} ,  \quad
V h < \frac{r}{4}, \quad
K_L h < \frac{r}{4}\quad  and  \quad 
K_L h < \ell(U)  
\label{hypo:h0}
\end{align}
where $\ell(U)$ is defined in Theorem \ref{thm:projection}  by replacing $C$ with $C(t_0)$. 
Note that for all $h<h_0$ and for all $t \in \tilde{J}$, $x^n(t) \in U$.

Recall that $x^n$ satisfies equation (\ref{inclu:amontrer}) on $J$, which is
$$ h^{-1} \Gamma_{x^n(\theta^n(t)),x^n(\tau^n(t))} +  f(\theta^n(t),x^n(\theta^n(t))) +\RR(h)\in \NN(C(\theta^n(t)),x^n(\theta^n(t))).$$
By Lemma \ref{lem:hypo} 
and thanks to the boundedness of $f$ and of the discretized velocities
(\ref{eq:vit}), there exist constants $\beta$ and  $ \A_U$ such that for all $c^n\in C(\theta^n(t))  \cap U$ satisfying $d(c^n, x^n(\theta^n(t))) \leq \rho(U)$, 
\begin{align*}
&  \langle h^{-1} \Gamma_{x^n(\theta^n(t)),x^n(\tau^n(t))} +
f(\theta^n(t),x^n(\theta^n(t)))+ \textrm{R}(h), 
\Gamma_{x^n(\theta^n(t)),c^n} \rangle_{T_{x^n(\theta^n(t))}M} \\
& \hspace{5cm} \leq \beta{\A_U } d(x^n(\theta^n(t)),c^n)^2.  
\end{align*}

So using a parallel transport with (\ref{eq:der}), 
\begin{align}
  &\langle -\dot x^n(t) + L_{x^n(\theta^n(t)) \to x^n(t)}\left[ f(\theta^n(t),x^n(\theta^n(t)))\right]+ \RR_2(h), \nonumber \\ 
  & \hspace{1cm} L_{x^n(\theta^n(t)) \to x^n(t)} \Gamma_{x^n(\theta^n(t)),c^n} \rangle_{T_{x^n(t)}M} \leq \beta {\A_U  }  d(x^n(\theta^n(t)),c^n)^2, 
\label{maj}
\end{align}
where $ | \RR_2(h) | = \OO(h)$.
Now in order to apply usual arguments in the Hilbertian context, we use the chart $\Phi$ to work in $H$.  So we define
$y^n = \Phi(x^n)$. By this way, $\dot y^n$ and $y^n$ are bounded
sequences and $y^n$ strongly converges to $y:=\Phi(x)$. Then we know
that up to a subsequence, we can assume that $\dot y^n$ $*-$weakly
converges to $\dot y$ in  $L^\infty(\tilde{J},H)$ which implies that $\dot
y^n$ weakly converges to $\dot y$ in  $L^1(\tilde{J},H)$ since $\tilde{J}$ is bounded. 
Consequently by Mazur's Lemma, there exists a subsequence $z^n \in L^1(\tilde{J},H)$ satisfying for almost every $t\in \tilde{J}$
\be{eq:zn} z^n(t) \in \textmd{Conv} \left( \dot y^k (t), \ k \geq n \right)  \xrightarrow[n \to \infty ]{} \dot y(t) \ee
 where \textrm{Conv} stands for the convex combinations and
\be{eq:zn2} z^n \xrightarrow[n \to \infty ]{} \dot y \textmd{ in } L^{1} (\tilde{J},H). \ee

We now look for proving for a.e. $t$ (when (\ref{eq:zn}) holds) and all $c \in C(t)\cap U$ satisfying $ d(x(t),c) <r/4$:
$$  \langle -\dot x(t) +f(t,x(t)) ,  \Gamma_{x(t),c} \rangle_{T_{x(t)}M} \leq \beta  {\A_U } d(x(t),c)^2.$$
Let $h < h_0$, $t \in \tilde{J}$ (such that (\ref{eq:zn}) holds), $c \in C(t) \cap U$ satisfying  $ d(x(t),c) <r/4$.  Thus  $d_{C(\theta^n(t))}(c) \leq d_H (C(t), C(\theta^n(t))) \leq K_L h < \ell(U) $ by Assumption \ref{ass:C} and (\ref{hypo:h0}). Indeed, we recall that the quantity $\ell(U) $ is the same for all the sets $C(t), t \in [0,T]$ thanks to Remark \ref{rem:constantell} and Assumption \ref{ass:C}. 
 Consequently  Theorem \ref{thm:projection} implies that there is $ c^n \in \PPP_{C(\theta^n(t))}(c)$. Since 
 \be{majccn}
 d(c,c^n) = d_{C(\theta^n(t))}(c) \leq  K_L h < r/4,  
\ee
 we deduce from (\ref{hypo:J}) that 
\begin{align*}
d(c^n , x(t_0)) \leq d(c^n , c)  + d(c, x(t)) + d(x(t), x(t_0)) \leq 3r/4 < r. 
\end{align*}
In other words, $c^n$ belongs to $U$. 
Moreover with (\ref{hypo:h0}) we obtain
\begin{align}
 d(c^n , x^n(\theta^n(t))) & \leq  d(c^n , c)  + d(c, x(t))+ d(x(t), x^n(t)) + d(x^n(t),x^n(\theta^n(t) ) \nonumber \\
 & < r/2 + r/4 + Vh <r < \rho(U) 
\label{maj1}
\end{align}
and 
\begin{align}
 d(c , x^n(\theta^n(t))) & \leq  d(c, x(t))+ d(x(t), x^n(t)) + d(x^n(t),x^n(\theta^n(t) ) \nonumber\\
 & < 3r/4 <\rho(U). 
\label{maj2}
\end{align}
Firstly, the inequality \eqref{maj} is satisfied thanks to \eqref{maj1}. Secondly, 
\begin{align*}
 & \left| \langle L_{x^n(\theta^n(t)) \to x^n(t)} \left[ f(\theta^n(t),x^n(\theta^n(t)))\right]+ \RR_2(h) \right. , \vspace{6pt}\\
&  \hspace{1cm} \left. - L_{x^n(\theta^n(t)) \to x^n(t)} [\Gamma_{x^n(\theta^n(t)),c^n} - \Gamma_{x^n(\theta^n(t)),c}] \rangle \right|  \vspace{6pt}\\ 
 & \hspace{2cm} \leq ( \|f\|_\infty +\OO(h)) | L_{x^n(\theta^n(t)) \to x^n(t)} [\Gamma_{x^n(\theta^n(t)),c^n} - \Gamma_{x^n(\theta^n(t)),c}] | \vspace{6pt} \\ 
 & \hspace{2cm} \leq ( \|f\|_\infty +\OO(h)) | \exp^{-1}_{x^n(\theta^n(t))} (c^n)-\exp^{-1}_{x^n(\theta^n(t))} (c) | \vspace{6pt}\\
 & \hspace{2cm} \leq ( \|f\|_\infty +\OO(h)) C_e(U) d(c^n,c)  \textmd{ according to } \eqref{maj1} , \eqref{maj2}  \textmd{ and Assumption } \ref{ass:ic} \vspace{6pt} \\
 & \hspace{2cm} \leq ( \|f\|_\infty +\OO(h))  C_e(U)  d_H(C(t),C(\theta^n(t))) = \OO(h)  \textmd{ thanks to Assumption  \ref{ass:C}}. 
 \end{align*}
Thirdly, using  \eqref{majccn}, we have $$d(x^n(\theta^n(t)),c) \leq  d(x^n(\theta^n(t)),c^n) + d(c, c^n) \leq  d(x^n(\theta^n(t)),c^n) + K_L h  .$$ 
Hence with \eqref{maj1},
\begin{align*}
d(x^n(\theta^n(t)),c)^2 \leq  d(x^n(\theta^n(t)),c^n)^2  +2 K_L h \rho(U) + (K_L h)^2.
\end{align*}
The previous inequality always holds if $c^n$ and $c$ are switched (using  \eqref{maj2}) and as a consequence 
$$  d(x^n(\theta^n(t)),c^n)^2 -d(x^n(\theta^n(t)),c)^2 =\OO(h).$$  
Hence \eqref{maj} becomes
\begin{align*}
  & \langle -\dot x^n(t) + L_{x^n(\theta^n(t)) \to x^n(t)}\left[ f(\theta^n(t),x^n(\theta^n(t)))\right]+ \RR_2(h),   L_{x^n(\theta^n(t)) \to x^n(t)} \Gamma_{x^n(\theta^n(t)),c} \rangle_{T_{x^n(t)}M} \vspace{6pt} \\
  & \hspace{5cm} \leq \beta {\A_U  } d(x^n(\theta^n(t)),c)^2 +\OO(h).
\end{align*}
For $n \to \infty$, since $f$ is Lipschitz continuous on $TM$ and with (\ref{eq:lim}), we have as previously
\begin{align*}
 & \langle  L_{x^n(\theta^n(t)) \to x^n(t)}\left[ f(\theta^n(t),x^n(\theta^n(t)))\right]+ \RR_2 (h), - L_{x^n(\theta^n(t)) \to x^n(t)}\Gamma_{x^n(\theta^n(t)),c} \rangle_{T_{x^n(t)}M} \\ 
 & \hspace{5cm} + \langle f(t,x(t)),   \Gamma_{x(t),c} \rangle_{T_{x(t)}M} = \OO(h^{\frac{1}{2}}).
\end{align*}
Consequently, we have for large enough $k$ 
\begin{align*} 
 & \langle -\dot x^k(t) ,  L_{x^k(\theta^k(t)) \to x^k(t)} \Gamma_{x^k(\theta^k(t)),c} \rangle_{T_{x^k(t)}M} \\
 & \hspace{2cm} \leq \beta {\A_U } d(x^k(\theta^k(t)),c)^2 - \langle f(t,x(t)),  \Gamma_{x(t),c} \rangle_{T_{x(t)}M}+\OO(h^{\frac{1}{2}}), 
\end{align*}
which means
\begin{align*}
 & \left \langle  -D\Phi(x^k(t))^{-1} [\dot y^k(t)]  ,  L_{x^k(\theta^k(t)) \to x^k(t)} \Gamma_{x^k(\theta^k(t)),c} \right \rangle_{T_{x^k(t)}M} \\
 & \hspace{2cm} \leq  \beta {\A_U } d(x(t),c)^2 - \langle f(t,x(t)),  \Gamma_{x(t),c} \rangle_{T_{x(t)}M}+\OO(h^{\frac{1}{2}}) 
\end{align*}
because $ d(x^k(\theta^k(t)),c) \leq d(x(t),c)  + \OO(h^{\frac{1}{2}})$ 
and so
\begin{align*}
 &  \left  \langle - \dot y^k(t)  , \left[D\Phi(x^k(t))^{-1}\right]^* L_{x^k(\theta^k(t)) \to x^k(t)} \Gamma_{x^k(\theta^k(t)),c}  \right \rangle_{H} \\
 & \hspace{2cm} \leq \beta {\A_U } d(x(t),c)^2 - \langle f(t,x(t)),  \Gamma_{x(t),c} \rangle_{T_{x(t)}M}+\OO(h^{\frac{1}{2}}).
\end{align*}
In addition 
$$\left[D\Phi(x^k(t))^{-1}\right]^* \left( L_{x^k(\theta^k(t)) \to x^k(t)} \Gamma_{x^k(\theta^k(t)),c}\right ) = \left[D\Phi(x(t))^{-1}\right]^*( \Gamma_{x(t),c})+\OO(h^{\frac{1}{2}}).$$
Indeed setting $ w_k =  L_{x^k(\theta^k(t)) \to x^k(t)} \Gamma_{x^k(\theta^k(t)),c}$, $\tilde{w_k} =  \Gamma_{x^k(t),c}$ and $w=  \Gamma_{x(t),c}  $, we have
\begin{align*}
& \left[D\Phi(x^k(t))^{-1}\right]^*(w_k) -  \left[D\Phi(x(t))^{-1}\right]^* (w) \\
&  \hspace{1cm}= \left[D\Phi(x^k(t))^{-1}\right]^*(w_k - \tilde{w_k} + \tilde{w_k} - L_{x(t) \to x^k(t)} w) \\
 &   \hspace{2cm}  +  \left[D\Phi(x^k(t))^{-1}\right]^* ( L_{x(t) \to x^k(t)} w)   - \left[D\Phi(x(t))^{-1}\right]^* (w) \\
&   \hspace{1cm} = \OO(h^{\frac{1}{2}}),
\end{align*}
by smoothness of the chart $\Phi$ and because $w_k - \tilde{w_k} = \OO(h)  $, $ \tilde{w_k} - L_{x(t) \to x^k(t)} w =  \OO(h^{\frac{1}{2}}) $. 
From this equality we deduce that 
$$  \langle -\dot y^k(t)  , \left[D\Phi(x(t))^{-1}\right]^* \Gamma_{x(t),c} \rangle_{H} \leq \beta  {\A_U } d(x(t),c)^2 - \langle f(t,x(t)),  \Gamma_{x(t),c} \rangle_{T_{x(t)}M}+\OO(h^{\frac{1}{2}}).$$
So applying (\ref{eq:zn}) and then (\ref{eq:zn2}) give for almost every time $t\in \tilde{J}$ and all $c\in U \cap C(t)$ satisfying $d(x(t),c)< r/4$,
$$  \langle -\dot y(t)  , \left[D\Phi(x(t))^{-1}\right]^* \Gamma_{x(t),c} \rangle_{H} \leq \beta {\A_U} d(x(t),c)^2 - \langle f(t,x(t)),  \Gamma_{x(t),c} \rangle_{T_{x(t)}M},$$
which is equivalent to
$$  \langle- \dot x(t)  + f(t,x(t)) ,  \Gamma_{x(t),c} \rangle_{T_{x(t)}M} \leq  \beta {\A_U } d(x(t),c)^2.$$
As a consequence, Proposition \ref{prop:hypomo} yields that for such time $t\in \tilde{J}$, 
\begin{align}
 -\dot x(t) +f(t,x(t)) \in \NN( C(t), x(t)).
\label{eq:incl}
\end{align}
We have shown that for every $t_0 \in J$, there is an open interval $\tilde{J}$ where inclusion \eqref{eq:incl} is satisfied almost everywhere. By a compactness argument, $J$ can be covered by a finite number of these intervals so finally \eqref{eq:incl} holds for $a. e. \, t \in J $. Furthermore by dividing $I= [0, T] $ into small intervals of length $|J|= \bar T$, we can again follow these five steps and obtain the same result for $a. e. \,  t \in I $.  In other words, the function $x$ (so built in the whole interval $I$) is a solution of (\ref{eq:sys}).

\mb {\bf Sixth step: } Uniqueness \\
Let $x_1$, $x_2$ be two solutions of (\ref{eq:sys}). Since $x_1 $ and $x_2$ belong to $W^{1,\infty}(I,M)$, for all $t \in  I, \, x_1(t), x_2(t)  \in \K$, 
where $\K$ is a bounded set of $M$. Furthermore setting $$ \displaystyle \alpha := \frac{\rho(\K)}{ 2\max(\|\dot{x_1}\|_\infty, \|\dot{x}_2\|_\infty)  } >0$$ we have for $i=1,2$ and for all $t \leq \alpha $, $d(x_i(t), x_0) \leq \rho(\K) /2$ and so $d(x_1(t),x_2(t)) \leq \rho(\K) $.\\ 
It comes from Lemma \ref{lem:hypo} that for a.e. $t \leq \alpha $, 
\be{ineg1} \langle- \dot{x_1}(t) + f(t, x_1(t)),  \Gamma_{x_1(t), x_2(t)} \rangle _{T_{x_1(t)}M} \leq   \A_\K  |\dot{x_1}(t) - f(t, x_1(t)) |  d(x_1(t),x_2(t))^2  , \ee
and 
\be{ineg2} \langle -\dot{x_2}(t) + f(t, x_2(t)), \Gamma_{x_2(t), x_1(t)} \rangle _{T_{x_2(t)}M} \leq   {\A_\K  |\dot{x_2}(t) - f(t, x_2(t)) | } d(x_1(t),x_2(t))^2 . \ee  
Moreover with (\ref{eq:symmetry}),  $L_{x_1(t)  \to x_2(t)}\left(\Gamma_{x_1(t),x_2(t)}\right) = -\Gamma_{x_2(t),x_1(t)}$. Thus (\ref{ineg1}) becomes
\begin{align}  
& \langle L_{x_1(t) \to  x_2(t)}  (\dot{x_1}(t) - f(t, x_1(t))),  \Gamma_{x_2(t), x_1(t)} \rangle _{T_{x_2(t)}M} \nonumber \\
 & \hspace{3cm} \leq   {\A_\K  |\dot{x_1}(t) - f(t, x_1(t)) | } d(x_1(t),x_2(t))^2 . \label{ineg3}
\end{align}
By summing (\ref{ineg2}) and (\ref{ineg3}), we have 
 \begin{align}  
 & \langle  L_{x_1(t) \to x_2(t)}  (\dot{x_1}(t) - f(t, x_1(t))) - (\dot{x_2}(t) - f(t, x_2(t))),  \Gamma_{x_2(t), x_1(t)} \rangle _{T_{x_2(t)}M} \nonumber \\
 & \hspace{3cm} \leq   \A_\K  F  d(x_1(t),x_2(t))^2 , \label{ineg4}
 \end{align}
where $F= 2 \|f\|_\infty + \|\dot{x_1}\|_\infty +\|\dot{x}_2\|_\infty $. 
Assumption \ref{ass:f} implies that $$ | f(t, x_2(t))- L_{x_1(t) \to x_2(t)}  ( f(t, x_1(t))) |\leq L_f d(x_1(t),x_2(t)). $$ As $|\Gamma_{x_2(t), x_1(t)}  | = d(x_1(t),x_2(t))$, 
we have 
\be{ineg5}  \langle f(t, x_2(t))- L_{x_1(t) \to x_2(t)}  ( f(t, x_1(t))) ,   \Gamma_{x_2(t), x_1(t)} \rangle \leq  L_f d(x_1(t),x_2(t))^2. \ee
It follows from (\ref{ineg4}) and (\ref{ineg5}) that 
\be{ineg6} \langle  L_{x_1(t) \to x_2(t)}  (\dot{x_1}(t)) - \dot{x_2}(t),   \Gamma_{x_2(t), x_1(t)} \rangle _{T_{x_2(t)}M} \leq  \left({\A_\K  F }+L_f \right) d(x_1(t),x_2(t))^2 ,  \ee

Furthermore 
$$\begin{array}{lll}
 \frac{d}{dt} (d(x_1(t),x_2(t))^2 &= &\langle \dot{x_1}(t), \nabla_x d^2(x_1(t),x_2(t)) \rangle_{T_{x_1(t)}M}  \vspace{6pt}\\
 & & \hspace{2cm}+ \langle \dot{x_2}(t), \nabla_y d^2(x_1(t),x_2(t)) \rangle_{T_{x_2(t)}M}  \vspace{6pt}\\
 &=&\langle \dot{x_1}(t), -2 \Gamma_{x_1(t), x_2(t)} \rangle_{T_{x_1(t)}M} \vspace{6pt} \\
& & \hspace{2cm} + \langle \dot{x_2}(t), - 2 \Gamma_{x_2(t), x_1(t)} \rangle_{T_{x_2(t)}M} \vspace{6pt} \\
 &=& 2\langle  L_{x_1(t) \to x_2(t)}  (\dot{x_1}(t)) - \dot{x_2}(t),   \Gamma_{x_2(t), x_1(t)} \rangle _{T_{x_2(t)}M}.
\end{array}$$
We deduce that $$ \frac{1}{2}\frac{d}{dt} (d(x_1(t),x_2(t))^2  \leq  \left( {\A_\K  F }+L_f \right) d(x_1(t),x_2(t))^2. $$ which is equivalent to 
 $$ \frac{d}{dt} (d(x_1(t),x_2(t))^2  \leq 2 \left( {\A_\K  F }+L_f \right) d(x_1(t),x_2(t))^2. $$ With the help of Gronwall Lemma, we conclude that $x_1=x_2$ in $[0, \alpha]$ since $x_1(0)=x_0=x_2(0)$ and obviously this equality holds in $[0,T]$ with the same arguments. 
\end{proof}


\bibliographystyle{plain}

\end{document}